\documentclass[10pt]{iopart}
\usepackage{iopams}

\usepackage[square,sort&compress,numbers]{natbib}
\expandafter\let\csname equation*\endcsname\relax
\expandafter\let\csname endequation*\endcsname\relax
\usepackage{mathtools}
\usepackage{amsthm}
\usepackage{graphicx}
\usepackage[font=small,labelfont=bf]{caption}
\usepackage{capt-of}
\usepackage{tabularx}
\usepackage{color}
\usepackage{pgfplots}
\usepackage{tikz}
\usetikzlibrary{patterns}

\newtheorem{thm}{Theorem}[section]
\newtheorem{prop}[thm]{Proposition}
\newtheorem{lemma}[thm]{Lemma}
\newtheorem{cor}[thm]{Corollary}
\newtheorem{rem}[thm]{Remark}

\theoremstyle{definition}
\newtheorem{example}[thm]{Example}
\newtheorem{definition}[thm]{Definition}
\newtheorem{assmpt}{Assumption}

\eqnobysec

\newcommand{\N}{\mathbb{N}}
\newcommand{\R}{\mathbb{R}}
\newcommand{\norm}[1]{\left\Vert #1 \right\Vert}

\renewcommand{\H}{{\cal H}}
\newcommand{\X}{{\cal X}}
\newcommand{\Y}{{\cal Y}}
\newcommand{\Z}{{\cal Z}}
\newcommand{\dualX}{{\cal X^*}}
\newcommand{\dom}{\mathrm{dom}}
\newcommand{\ran}{\mathrm{ran}}
\newcommand{\id}{\mathrm{id}}
\newcommand{\argmin}{\mathrm{argmin}}

\newcommand{\calN}{{\cal N}}
\newcommand{\calP}{{\cal P}}

\newcommand{\defi}{:=}
\newcommand{\fwd}{A}
\newcommand{\adj}{B}
\renewcommand{\d}{\mathrm{d}}
\newcommand{\expd}{\alpha}
\newcommand{\expr}{\beta}
\newcommand{\st}{\,:\,}
\newcommand{\sgn}{\mathrm{sgn}}

\newcommand{\dataX}{u^\dagger}
\newcommand{\proj}{\mathrm{proj}}
\newcommand{\prox}{\mathrm{prox}}
\newcommand{\tv}{\mathrm{TV}}
\newcommand{\bv}{\mathrm{BV}}

\newcommand{\LB}[1]{\textcolor{red}{#1}}
\renewcommand{\LB}[1]{#1}
\newcommand{\minrev}[1]{\textcolor{red}{#1}}
\renewcommand{\minrev}[1]{#1}

\begin{document}

\title[Solution paths of variational regularization methods for inverse problems]{Solution paths of variational regularization methods for inverse problems}

\author{Leon Bungert \& Martin Burger}

\address{Department Mathematik, Friedrich-Alexander Universit\"{a}t Erlangen-N\"{u}rnberg, Cauerstrasse 11, 91058 Erlangen, Germany.}
\ead{\{leon.bungert,martin.burger\}@fau.de}
\vspace{10pt}
\begin{indented}
\item[]\today
\end{indented}

\begin{abstract}
We consider a family of variational regularization functionals for a generic inverse problem, where the data fidelity and regularization term are given by powers of a Hilbert norm and an absolutely one-homogeneous functional, respectively, and the regularization parameter is interpreted as artificial time. We investigate the small and large time behavior of the associated solution paths and, in particular, prove finite extinction time for a large class of functionals. Depending on the powers, we also show that the solution paths are of bounded variation or even Lipschitz continuous. In addition, it will turn out that the models are ``almost'' mutually equivalent in terms of the minimizers they admit. Finally, we apply our results to define and compare two different nonlinear spectral representations of data and show that only one of it is able to decompose a linear combination of nonlinear eigenvectors into the individual eigenvectors. Finally, we also briefly address piecewise affine solution paths.
\end{abstract}

\section{Introduction}

A standard approach for approximating solutions of an ill-posed inverse problem
\begin{equation}\label{eq:IP}
\fwd u=f\tag{IP}
\end{equation}
with possibly noise-corrupted data $f$ consists in variational regularization. To this end, one typically aims at solving the optimization problem
\begin{equation}\label{mod:var_P}
\min_u \mathcal{D}(\fwd u,f)+t\mathcal{R}(u)\tag{P}
\end{equation}
where the {data fidelity} term~$\mathcal{D}$ enforces~$\fwd u$ to be close to~$f$ and the {regularization} functional~$\mathcal{R}$ incorporates prior knowledge about the solution (sparsity, smoothness, etc.) into the model. The real number~$t>0$ is typically referred to as {regularization parameter} and balances data fidelity and regularization. One of the most famous examples for~\eref{mod:var_P} within the field of mathematical imaging is the ROF denosing model~\cite{rudin1992nonlinear}
\begin{equation}\label{mod:ROF}\tag{ROF}
\min_{u\in\bv(\Omega)}\frac{1}{2}\norm{u-f}^2_{L^2(\Omega)}+t\,\tv(u).
\end{equation}
Here, $t$ should be chosen dependent on the noise level of $f$ to obtain a satisfyingly denoised image. In contrast, the parameter~$t$ can also be interpreted as an artificial time that steers the solution of~\eref{mod:var_P} from being under-regularized to over-regularized as time increases, or speaking in the ROF context, that successively and edge-preservingly smoothes~$f$ until a constant state is reached. In this manuscript we will refer to the maps $t\mapsto\{u_t\st u_t\text{ solves }\eref{mod:var_P}\}$ and $t\mapsto\{\fwd u_t\st u_t\text{ solves }\eref{mod:var_P}\}$ as \emph{solution path} and \emph{forward solution path}, respectively. Recently, this and similar evolutions, which can be viewed as a scale space representation of the input~$f$, have been used to define nonlinear spectral multiscale decompositions, e.g.~\cite{burger2015spectral,burger2016spectral,gilboa2013spectral,
gilboa2014total,gilboa2016nonlinear,gilboa2018nonlinear}. \LB{Hence, in this context the solution of \eref{mod:ROF} becomes interesting even if the data $f$ is not noisy at all.} Typically, these decompositions involve computing derivatives with respect to the parameter~$t$ of the (forward) solution path wherefore it is interesting to study its regularity. 

Furthermore, not only in the ROF model but also in general, a very popular choice for the data fidelity in~\eref{mod:var_P} is the squared norm of some Hilbert space whereas the regularization functional is often assumed to be absolutely one-homogeneous. However, there is often no substantial justification for preferring such models over others. In particular, one could consider arbitrary powers of a Hilbert space norm~$\norm{\cdot}$ and of an absolutely one-homogeneous functional~$J$ instead which leads to the \emph{weighted problem}
\begin{equation}\label{mod:weighted_var_P}\tag{wP}
\min_u \frac{1}{\expd}\norm{\fwd u-f}^\expd+\frac{t}{\expr}J(u)^\expr
\end{equation}
with weights~$\expd,\expr\geq1$. Note that the multiplicative scalings~$1/\expd$ and~$1/\expr$ do not restrict generality since they can be absorbed into~$t$. Indeed there are only few contributions in literature that consider general powers of norms (cf.~\cite{he2006iterative,belloni2011square} for a Hilbert norm with~$\expd=1$ and~\cite{schuster2012regularization} for error analysis for a Banach norm with fixed~$\expd\geq1$) or a different scaling of an absolutely one-homogeneous regularization functional~\cite{deswarte2018minimax}. While such modifications seem only minor at first glance and the resulting models will be equivalent for parameters~$t$ in a certain interval, we will see that outside this interval the qualitative behavior of the models differs significantly. In a nutshell, the models disintegrate into four classes, depending on whether~$\expd$ or~$\expr$ are larger or equal than~$1$. If both parameters equal~$1$, due to the homogeneity of~$J$, the corresponding problem~\eref{mod:weighted_var_P} becomes \emph{contrast invariant}, meaning that if~$u$ solves~\eref{mod:weighted_var_P} with some~$f$ then~$cu$ solves the problem where~$f$ is replaced by~$cf$ and~$c>0$.

Our precise setting in this paper is as follows: Let~$(\X,\norm{\cdot}_\X)$ be the dual space of an separable predual Banach space~$\Y$ and let~$(\H,\langle\cdot,\cdot\rangle)$ be a Hilbert space with norm $\norm{\cdot}_\H\defi\sqrt{\langle\cdot,\cdot\rangle}$. We consider a bounded linear forward operator~$\fwd:\X\to\H$ mapping between these spaces and denote by~$\calN(\fwd)$ and~$\ran(\fwd)$ its null-space and range. Let furthermore $J:\X \rightarrow \R_+\cup \{+\infty\}$ be an absolutely one-homogeneous, weak$^*$ lower semi-continuous, and proper convex functional, whose null-space and {effective domain} we denote by $\calN(J)\defi\{u\in\X\st J(u)=0\}$ and $\dom(J)\defi\{u\in\X\st J(u)<\infty\}$, respectively. 
For parameters~$\expd,\expr\geq1$,~$t\geq0$, and given data~$f\in\H$ we define functionals
\begin{align}\label{eq:functional}
	&E^{\expd,\expr}_t(u;f) \defi \frac{1}{\expd}\Vert \fwd u - f \Vert^\expd_\H + \frac{t}{\expr} J(u)^\expr,\quad u\in\X,
\end{align}
which we aim to minimize. If~$f\in\ran(\fwd)$, meaning that there exists~$\dataX\in\X$ with~$\fwd\dataX=f$, we assume that~$\dataX\notin\calN(J)$. This is the only interesting scenario since otherwise~$\dataX$ is a minimizer of~$E^{\expd,\expr}_t(\cdot;f)$ for any~$t \geq 0$.

The remainder of this work is organized as follows: We will perform a thorough analysis of the variational problem at hand in an infinite dimensional setting in \sref{sec:analysis}. A special emphasis will lie on the small and large time behavior and of the so called {solution path} and uniqueness of the {forward solution path}. Furthermore, we briefly demonstrate the equivalence of some classes of the models under consideration. Using these results, \sref{sec:regularity} will deal with regularity of the forward solution path depending on the weights~$\expd$ and~$\expr$. In \sref{sec:spectral} we will indicate how our results can be used to define nonlinear spectral representations. We undertake numerical experiments that illustrate our theoretical findings in \sref{sec:numerics} and conclude with some open questions. Basic notation and relevant notions from convex analysis, as well as fundamental properties of generalized orthogonal complements and projections with respect to the forward operator $\fwd$ are collected in the appendix. 

\section{Analysis of the variational problem}\label{sec:analysis}

In this section we will provide a basic analysis of the variational problem of minimizing~\eref{eq:functional}. We start with fixed~$t$ and then proceed towards the behaviour of the solution path for small respectively large~$t$, which can allow for exact penalization respectively finite time extinction.

\subsection{Basic properties of the variational problem} 
 
In the following, we make three assumptions, related to the forward operator~$A$ and its interplay with the regularization functional~$J$ which we make use of throughout this manuscript:
\begin{assmpt}\label{assmpt:A-norm}
$\norm{u}_\fwd\defi\norm{\fwd u}_\H$ is a norm on~$\calN(J)$ which is equivalent to the restriction of~$\norm{\cdot}_\X$ to~$\calN(J)$.
\end{assmpt}

Note that for Assumption~\ref{assmpt:A-norm} to hold it is sufficient to have~$\calN(J)\cap\calN(\fwd)=\{0\}$ and $\dim\calN(J)<\infty$ together with an appropriate definition of~$\X$ which is satisfied in most cases. The second assumption is a generalized Poincar\'{e} inequality which assures a weaker form of coercivity of~$J$. To this end we define the map
\begin{align}\label{eq:orth_proj}
\calP^\fwd:
\begin{cases}
&\H\rightarrow\X,\\
&f\mapsto\calP^\fwd(f)\defi\argmin_{u\in\calN(J)}\norm{\fwd u - f}_\H,
\end{cases}
\end{align}
whose well-definedness and important properties are proved in \sref{sec:orth_proj} of the appendix. We call this map the $\fwd$-\emph{orthogonal projection} onto the null-space of~$J$.

\begin{assmpt}\label{assmpt:poincare}
There is~$C>0$ such that 
$$\norm{u-\calP^\fwd(\fwd u)}_\X\leq C J(u),\quad\forall u\in\X.$$
\end{assmpt} 
Apart from guaranteeing coercivity, this assumption will be utilized to study the small and large time behavior of the solution path.

\begin{assmpt}\label{assmpt:w*-w-c}
The operator~$\fwd$ is weak$^*$-to-weak continuous, that is if~$(u_k)\subset\X$ is a sequence which weakly$^*$ converges to some~$u\in\X$, then~$(\fwd u_k)$ weakly converges to~$\fwd u$ in~$\H$.
\end{assmpt}
This assumption is guaranteed if~$\fwd=\adj^*$ with some bounded linear operator~$\adj:\H\to\Y$. However, in some cases it is not obvious how to ensure this condition. In the following remark we demonstrate how an appropriate choice of the space $\X$ can accomplish this.

\begin{rem}\label{rem:predual_BV}
In most cases the space~$\X$ is solely determined by the regularization functional, but in some very mildly ill-posed cases the data fidelity needs to be taken into account as well in order to satisfy the assumptions. The canonical case is indeed~$\tv$ in multiple dimensions.  We define~$\X\defi\bv\cap L^2$ with norm \mbox{$\|\cdot\|_\X\defi\|\cdot\|_\bv+\|\cdot\|_{L^2}$}, choose~$\H=L^2$, and let~$\fwd$ be the continuous embedding operator. A predual of~$\X$ is given by~$\Y\defi \Z+L^2$ where~$\Z^*=\bv$. Since weak$^*$ convergence in~$\X$ implies in particular weak~$L^2$-convergence, the embedding~$\X\hookrightarrow\H$ is weak$^*$-to-weak continuous. More general, it can be checked that the dual of a sum of Banach spaces equals the intersection of the duals. 
\end{rem}

Now we provide some basic results concerning the minimization problem for the energy functional~$E^{\expd,\expr}_t(\cdot;f)$. \LB{We start with an existence result which follows by standard arguments using Assumptions~\ref{assmpt:A-norm}-\ref{assmpt:w*-w-c}.}

\begin{thm}[Existence of minimizers]\label{thm:existence}
Let Assumptions~\ref{assmpt:A-norm}-\ref{assmpt:w*-w-c} hold. For each~$f \in \H$,~$t >0$, and~$\expd,\expr\geq 1$ there exists a minimizer~$u_t$ of~$E^{\expd,\expr}_t(\cdot;f)$. If $\fwd$ is injective and $\expd>1$ this minimizer is unique.
\end{thm}

Now we turn to optimality conditions for minimizers. In some of the following statements we will utilize the \emph{range condition} 
\begin{equation}\label{eq:range_cond}\tag{RC}
\exists\dataX\in\dom(J)\st\fwd\dataX=f
\end{equation}
which applies if the inverse problem \eref{eq:IP} possesses a (possibly not unique) solution. For convenience we also define~$B^\H_1\defi\{q\in\H\;:\;\norm{q}_\H\leq1\}$.
\begin{thm}[Optimality conditions]\label{thm:opt_cond}
Let~$t > 0$ and~$\expd,\expr\geq 1$,~$u_t$ be a minimizer of~$E^{\expd,\expr}_t(\cdot;f)$. We distinguish between two cases: If~$u_t=\dataX$ for some $\dataX$ which satisfies~\eref{eq:range_cond}, then $\expd=1$ holds necessarily and there is~$q\in B^\H_1$ such that 
\begin{equation}\label{eq:opt_conf_f}
	p_t\defi -\frac{\fwd^*q}{t J(\dataX)^{\expr-1}}\in\partial J(\dataX).
\end{equation}
If $u_t$ is such that $\fwd u_t\neq f$, it holds
\begin{equation}\label{eq:opt_cond_u}
	p_t \defi \frac{\fwd^*(f -\fwd u_t)}{t \Vert\fwd u_t - f\Vert_\H^{2-\expd}J(u_t)^{\expr-1}} \in \partial J(u_t),
\end{equation}
where we use the convention~$0^0=1$ if~$\expr=1$ and~$J(u_t)=0$.
\end{thm} 
\begin{proof}
Standard results of subgradient calculus~\cite{schuster2012regularization} allow us to calculate the subdifferential of the energy functional~\eref{eq:functional}. Note in particular that~$u\mapsto\frac{1}{\expd}\Vert \fwd u -f \Vert^\expd_\H$ is continuous, thus the subgradients of~$E^{\expd,\expr}_t(\cdot;f)$ are given by the sum of subgradients of~$\frac{1}{\expd}\Vert \fwd\cdot - f\Vert^\expd_\H$ and~$\frac{t}{\expr}J(\cdot)^\expr$. 
By the chain rule for subdifferentials, see~\cite{bauschke2017convex} for instance, the subdifferential of~$E^{\expd,\expr}_t(\cdot;f)$ in~$u\in$~$\dom(J)$ reads
\begin{align}\label{eq:subdiff_E}
\partial E^{\expd,\expr}_t(u;f)=\norm{\fwd u-f}_\H^{\expd-1}\partial\left(\norm{\fwd u-f}_\H\right)+tJ(u)^{\expr-1}\partial J(u)
\end{align}
and for any~$q\in\H$ it holds
$$\partial\norm{q}_\H=
\begin{cases}
B^\H_1,&q=0,\\
\frac{q}{\norm{q}_\H},&q\neq 0.
\end{cases}$$
Hence, the optimality condition for~$\dataX$ and~$\expd>1$ reads
$$0\in\partial E^{\expd,\expr}_t(\dataX;f)=tJ(\dataX)^{\expr-1}\partial J(\dataX)$$
which contradicts~$t>0$ since~$J(\dataX)\neq 0$, by assumption. Therefore,~$\dataX$ cannot be a minimizer for~$\expd>1$. Similarly, any minimizer~$u_t$ for~$\expr>1$ satisfies~$u_t\notin\calN(J)$ since otherwise~$f=\fwd u_t$ held true due to \eref{eq:subdiff_E}. This would contradict our non-triviality assumption on the data. Equations~\eref{eq:opt_conf_f} and~\eref{eq:opt_cond_u} follow from rewriting the condition~$0\in\partial E_t^{\expd,\expr}(u_t;f)$.
\end{proof}
\begin{rem}\label{rem:opt_cond}
Due to convexity of~$E_t^{\expd,\expr}(\cdot;f)$, conditions~\eref{eq:opt_conf_f} and~\eref{eq:opt_cond_u} are also sufficient for optimality.
\end{rem}


\LB{As we have seen in Theorem~\ref{thm:existence}, minimizers are unique under stronger assumptions on the forward operator $\fwd$. However, in the general case one can still prove that the norm of the residual and the value of the regularizer of minimizers are uniquely determined for~$\expd>1$ or~$\expr>1$. The statement follows from standard arguments, is implicitly used in several proofs in the literature, however, it is usually not stated clearly, despite being a result of interest.}
\begin{thm}[Uniqueness of residuals]\label{thm:uniqueness_res}
Let~$\Phi,\Psi:[0,\infty)\rightarrow[0,\infty)$ be increasing and convex,~$J:\X\to\R\cup\{\infty\}$ be convex and proper, and~$u,v\in\X$ be two minimizers of~$E_t(\cdot)\defi \mathcal{D}(\cdot)+t\mathcal{R}(\cdot)$ where~$\mathcal{D}(\cdot)\defi \Phi(\norm{\fwd \cdot-f})$,~$\mathcal{R}(\cdot)\defi \Psi(J(\cdot))$, and~$t>0$. If~$\Phi$ or~$\Psi$ is strictly convex, then~$\norm{\fwd u-f}_\H=\norm{\fwd v-f}_\H$ and~$J(u)=J(v)$.
\end{thm}

\begin{rem}
With a little abuse of notation we introduce the following maps
\begin{align}\label{eq:residuum}
&R:(0,\infty)\rightarrow[0,\infty),\;t\mapsto R(t)\defi\norm{\fwd u_t-f}_\H,\\\label{eq:regulariser}
&J:(0,\infty)\to[0,\infty),\;t\mapsto J(t)\defi J(u_t),
\end{align}
where~$u_t$ is a minimizer of~$E_t^{\expd,\expr}(\cdot;f)$. Note that we suppress the dependency of~$R$ on~$\expd$ and~$\expr$ for concise notation. By Theorem~\ref{thm:uniqueness_res} the maps $R$ and $J$ are well-defined for $\expd>1$ or $\expr>1$. If $\expd,\expr=1$, we will use the same expressions for minimizers of~$E_t^{1,1}(\cdot;f)$ although their values will depend on the individual minimizer, in general.
\end{rem}

A fairly well-known property is that the residual map~$t\mapsto R(t)$ is monotonously increasing whereas the regularizer map~$t\mapsto J(t)$ decreases monotonously. \LB{The proof works precisely as in \cite{burger2013guide} which deals with the case $J=\tv$.}
\begin{lemma}\label{lem:monot_res}
Let~$0<s<t$ and~$u_s,u_t$ denote minimizers of~$E_s^{\expd,\expr}(\cdot;f)$ and~$E^{\expd,\expr}_t(\cdot;f)$, respectively. Then it holds $R(s)\leq R(t)$ and $J(s)\geq J(t)$, where the inequalities are strict if minimizers are unique. 
\end{lemma}

\subsection{Behaviour for small time}
\label{sec:small_time}
Obviously, for $t=0$ any $\dataX$ fulfilling~\eref{eq:range_cond} is a minimizer of~$E_0^{\expd,\expr}(\cdot;f)$. In this section we consider the special case~$\expd=1$ where such~$\dataX$ can be a solution for small~$t>0$, as well. \LB{This phenomenon is called \emph{exact penalization} and has been introduced in \cite{burger2004convergence}. Due to the regularizing effect of the minimzation of \eref{eq:functional}, certainly this exotic behavior can only occur if the datum $f$ is noise-free. Although this situation might be of limited relevance in practical situations, it is important to understand and characterize exact penalization from a theoretical perspective, e.g. in order to obtain convergence rates (cf.~\cite{hofmann2007convergence,anzengruber2014regularization}).} We shall assume that~\eref{eq:range_cond} holds and assume that there is some $u^\dagger$ which also fulfills the following \emph{source condition}: 
\begin{equation}\label{eq:source_cond}\tag{SC}
\exists q\in\H\st\fwd^*q\in\partial J(u^\dagger).
\end{equation}
Needless to say, since~$J(\dataX)\neq 0$, any such~$q$ fulfilling~\eref{eq:source_cond} is also different from zero. Furthermore, according to \cite{burger2004convergence} such $\dataX$ fulfills range and source condition if and only if it is a $J$-\emph{minimizing solution} of the forward problem \eref{eq:IP}, i.e., $J(\dataX)\leq J(u)$ for all $u\in\X$ with $\fwd u=f$. In particular, the (positive) value $J(\dataX)$ does not depend on the choice of $\dataX$ and will be denoted by $J_{\min}$, in the sequel. It is obvious from the optimality condition~\eref{eq:opt_conf_f} that~\eref{eq:source_cond} is necessary for~$\dataX$ being a minimizer for~$t>0$. Indeed, the source condition is also sufficient. To show this, we start with the following lemmas.
\begin{lemma}\label{lem:well_def_s*}
Let conditions~\eref{eq:range_cond} and~\eref{eq:source_cond} hold true. Then~$s_*$ given by
\begin{equation}\label{eq:s_*}
s_* \defi \inf_{\substack{\dataX\in\X:\\\eref{eq:range_cond},\,\eref{eq:source_cond}\text{ hold}}}\inf\left\lbrace\norm{q}_\H\st q\in\H,\,\fwd^*q\in\partial J(\dataX)\right\rbrace
\end{equation}
fulfills~$0<s_*<\infty$.
\end{lemma}
\begin{proof}
Let us now assume that there is a sequence $(\dataX_k)\subset\X$ fulfilling conditions~\eref{eq:range_cond} and~\eref{eq:source_cond} and a corresponding sequence of source elements $(q_k)\subset\H$ with $\fwd^* q_k\in\partial J(\dataX_k)$ for all $k$ such that $\lim_{k\to\infty}\norm{q_k}_\H=0$. In this case we calculate
$$0<J_{\min}=J(\dataX_k)=\langle\fwd^* q_k,\dataX_k\rangle=\langle q_k,f\rangle\leq\norm{f}_\H\norm{q_k}_\H\to 0,\quad k\to\infty$$
which is a contradiction.

Finally, assumptions~\eref{eq:range_cond} and~\eref{eq:source_cond} imply that the admissible sets in~\eref{eq:s_*} are non-empty and hence~$s_*<\infty$. 
\end{proof}

\begin{lemma}\label{lem:s_*_attained}
Under the conditions of Lemma~\ref{lem:well_def_s*} the infimum is attained, i.e., there is~$\hat{u}\in\dom(J)$ fulfilling $\fwd\hat{u}=f$ and~$\hat{q}\in\H$ with~$\fwd^*\hat{q}\in\partial J(\hat{u})$ such that~$\norm{\hat{q}}_\H=s_*$.
\end{lemma}
\begin{proof}
Let~$(\dataX_k)\subset\X$ fulfilling~\eref{eq:range_cond} and $(q_k)\subset\H$ such that~$\fwd^*q_k\in\partial J(\dataX_k)$, for every~$k\in\N$, be a minimizing sequence \minrev{for \eref{eq:s_*}, meaning that $\lim_{k\to\infty}\norm{q_k}_\H=s_*$}. By Assumption~\ref{assmpt:poincare} we infer
\begin{align*}
\norm{\dataX_k-\calP^\fwd(\fwd \dataX_k)}_\X\leq C J(\dataX_k)=CJ_{\min}<\infty,\quad\forall k\in\N.
\end{align*}
Hence, $\left(\dataX_k-\calP^\fwd(\fwd \dataX_k)\right)$ is bounded in $\X$ and admits a subsequence (denoted with the same index) which weakly$^*$ converges to some $h\in\X$. As $\calP(\fwd\dataX_k)=\calP^\fwd(f)$ holds for all $k\in\N$, we obtain that $(\dataX_k)$ converges to $\hat{u}\defi h+\calP^\fwd(f)$. Using again that $\fwd \dataX_k=f$, this implies that $f=\fwd\hat{u}$. Furthermore, by the lower semi-continuity of $J$, we infer that $\hat{u}\in\dom(J)$. Hence, we have shown that the limit of $(\dataX_k)$ fulfills~\eref{eq:range_cond}.

Similarly, being a minimizing sequence, $(q_k)$ is bounded in~$\H$ and a subsequence weakly converges to some $\hat{q}\in\H$. It holds (after another round of subsequence refinement)
\begin{align*}
\langle\fwd^*\hat{q},\hat{u}\rangle=\langle\hat{q},f\rangle=\lim_{k\to\infty}\langle q_k,f\rangle=\lim_{k\to\infty}\langle\fwd^*q_k,\dataX_k\rangle=\lim_{k\to\infty}J(\dataX_k)\geq J(\hat{u}),
\end{align*}
using the lower semi-continuity of $J$. On the other hand, one clearly has $J(\dataX_k)=J_{\min}\leq J(\hat{u})$, for all $k\in\N$ since $\hat{u}$ satisfies~\eref{eq:range_cond}. This shows $\langle\fwd^*\hat{q},\hat{u}\rangle=J(\hat{u}).$ Furthermore, from
$$\langle\fwd^* q_k-\fwd^*\hat{q},u\rangle=\langle q_k-\hat{q},\fwd u\rangle,\quad\forall u\in\X,$$
and the weak convergence of $(q_k)$ to $\hat{q}$ we infer that $(\fwd^*q_k)$ weakly$^*$ converges to $\fwd^*\hat{q}$ in $\dualX$. Since the sequence $(\fwd^*q_k)$ lies in $\partial J(0)$ which is weakly$^*$ closed (cf.~\cite{ekeland1999convex}), also $\fwd^*\hat{q}\in \partial J(0)$ holds. \minrev{Using \eref{eq:subdiff}}, we have shown that $\fwd^*\hat{q}\in\partial J(\hat{u})$, as desired. \minrev{Remains to show $\norm{\hat{q}}_\H=s_*$. The definition of $s_*$ and the lower semi-continuity of the Hilbert norm implies $s_*\leq\norm{\hat{q}}_\H\leq\lim_{k\to\infty}\norm{\hat{q_k}}_\H=s_*$ by the assumption that $(q_k)$ is a minimizing sequence. This concludes the proof.}
\end{proof}

As a consequence of Lemmas~\ref{lem:well_def_s*} and~\ref{lem:s_*_attained} we obtain

\begin{thm} \label{thm:fsolution}
Under the conditions of Lemma~\ref{lem:well_def_s*} there is a minimizer $u_t$ of~$E^{1,\expr}_t(\cdot;f)$ fulfilling $\fwd u_t=f$ if and only if~$t\leq t_*$, where $t_*\defi{ J_{\min}^{1-\expr}}/{s_*}$ and $s_*$ is given by~\eref{eq:s_*}. 
\end{thm}
\begin{proof}
Let~$t\leq t_*$ and choose $\hat{u}\in\H$ and $\hat{q}\in\H$ as in the proof of Lemma~\ref{lem:s_*_attained}. Defining ${p}\defi\fwd^*\hat{q}$ and $q\defi-tJ_{\min}^{\expr-1}\hat{q}$ we find that~$\fwd^*q+tJ_{\min}^{\expr-1}p=0$ and~$\norm{q}_\H=tJ_{\min}^{\expr-1}\norm{\hat{q}}_\H\leq t_*J_{\min}^{\expr-1}s_*=1$. Consequently, by the optimality conditions (cf.~Theorem~\ref{thm:opt_cond} and Remark~\ref{rem:opt_cond}) it follows that~$\hat{u}$ is a minimizer of~$E^{1,\expr}_t(\cdot;f)$. 

On the other hand, let~$\dataX$, fulfilling \eref{eq:range_cond} and \eref{eq:source_cond}, be a minimizer of~$E_t^{1,\expr}(\cdot;f)$. By \eref{eq:opt_conf_f} from Theorem~\ref{thm:opt_cond} there are~$q\in B_1^\H$ and~$0\neq p\in\partial J(\dataX)$ such that~$\fwd^*q+tJ_{\min}^{\expr-1}p=0$ or equivalently $p=-\fwd^*q/(tJ_{\min}^{\expr-1})$. Hence, it holds by definition of $s_*$ that $\norm{q/(tJ_{\min}^{\expr-1})}\geq s_*$ which implies $t\leq 1/(s_*J_{\min}^{\expr-1})=t_*$.
\end{proof}
Note that in the second part of the proof the source condition follows directly from the optimality condition and does not have to be imposed. Next we show that for $t < t_*$ the forward solution path (and hence the residual) is uniquely determined:

\begin{thm}
Let \eref{eq:range_cond} and \eref{eq:source_cond} hold. Every minimizer~$u_t$ of~$E_t^{1,\expr}(\cdot;f)$ for~$0 < t < t_*$ fulfills~$\fwd u_t=f$. 
\end{thm}
\begin{proof}
Suppose~$u_t$ is a minimizer for~$0 < t < t_*$ and~$\fwd u_t\neq f$. Then
$$ \Vert \fwd u_t - f\Vert_\H + \frac{t}{\expr} J(u_t)^\expr \leq \frac{t}{\expr} J(\dataX)^\expr,$$
where $\dataX$ fulfills range and source condition. Hence, multiplication with~$t_*/t>1$ yields
$$  \Vert\fwd u_t - f\Vert_\H + \frac{t_*}{\expr} J(u_t)^\expr < \frac{t_*}{\expr} J(\dataX)^\expr$$
which contradicts that~$\dataX$ is a minimizer of~$E_{t_*}^{1,\expr}(\cdot;f)$.
\end{proof}

In order to maintain a concise notation, for the rest of this manuscript we will define $t_*\defi 0$ if $\expd>1$ or if conditions~\eref{eq:range_cond} and~\eref{eq:source_cond} do not hold.

\subsection{Behaviour for large time}
\label{sec:large_time}
\LB{It is well-known that for increasing parameters $t$ in the ROF model, the solution approaches the mean value of the data. Similarly, if the regularization functional is given by a norm, the solution will approach zero. However, if a non-trivial forward operator mapping between two distinct spaces $\X$ and $\H$ and a general regularization functional are involved, the situation becomes unclear.} Hence, we investigate the behavior of minimizers $u_t$ of our general functional $E_t^{\expd,\expr}(\cdot;f)$ for~$t$ sufficiently large and we expect that~$u_t$ behaves the like a solution of
\begin{align}\label{eq:long_time_prob}
\inf_{u\in\calN(J)}\norm{\fwd u-f}_\H,
\end{align}
which is the $\fwd$-orthogonal projection of $f$ onto $\calN(J)$, introduced in~\eref{eq:orth_proj}. We refer to \sref{sec:orth_proj} of the appendix for further details. \LB{Note that the projection is not always as trivial as in the introductory examples of this section. In particular, if the null-space of the functional becomes bigger, as it is the case for higher order regularizations like total generalized variation \cite{bredies2010total}, it does not even admit a closed form. Furthermore, it is not obvious whether or not minimizers $u_t$ converge to the solution of \eref{eq:long_time_prob} for \emph{finite} parameters $t$. Note that the study of extinction times is also closely related nonlinear spectral theory (cf.~\cite{gilboa2014total} and \sref{sec:spectral}) since it relates to the eigenvalues contained in the data $f$. In a nutshell, the parts of the data which correspond to small eigenvalues extinct quickly, whereas the low eigenvalue components persist until a larger time $t$.} 

\begin{rem}
Note that for~$\X=\H$ and~$\fwd=\id$ it holds that~$\calP^\fwd=\calP$, i.e., the minimizer of~\eref{eq:long_time_prob} coincides with the orthogonal projection on~$\calN(J)$ which fulfills~$\langle f-\calP(f),\calP(f)\rangle=0$.
\end{rem}

But even in our more general setting one can obtain properties for~$\calP^\fwd$ which resemble the classical ones for orthogonal projections in Hilbert spaces. These are subsumed in Proposition~\ref{prop:properties_proj} and will be needed to obtain finite extinction time of minimizers of~$E_t^{\expd,\expr}(\cdot;f)$ with~$\expr=1$, meaning that there is~$T>0$ such that all minimizers for~$t>T$ coincide with~$\calP^\fwd(f)$. However, first we will prove a weaker statement, namely that minimizers of~$E_t^{\expd,\expr}(\cdot;f)$ converge to~$\calP^\fwd(f)$ as~$t$ tends to infinity. 

\begin{thm}
Let $(t_k)\subset(0,\infty)$ be a sequence tending to infinity and~$u_{t_k}$ be a minimizer of~$E^{\expd,\expr}_{t_k}(\cdot;f)$. Then~$(u_{t_k})$ weakly$^*$ converges to~$u_\infty\defi\calP^\fwd(f)$ in $\X$ as~$k\to\infty$. 
\end{thm} 
\begin{proof}
Since~$E^{\expd,\expr}_{t_k}(u_{t_k};f) \leq E^{\expd,\expr}_{t_k}(u_\infty;f)$, we obtain
\begin{align}
\frac{1}{\expr} J(u_{t_k})^\expr \leq \frac{\Vert \fwd u_\infty - f \Vert^\expd_\H}{\expd~t_k},\qquad\text{and}\qquad\label{ineq:estimates}
\norm{\fwd u_{t_k}-f}_\H\leq\norm{\fwd u_\infty-f}_\H,
\end{align} 
and, in particular,~$J(u_{t_k}) \rightarrow 0$ as~$k\to\infty$. Furthermore, for~$k$ large enough it holds $E_1^{\expd,\expr}(u_{t_k};f)\leq E_{t_k}^{\expd,\expr}(u_{t_k};f)$ and since the functional~$E_1^{\expd,\expr}(\cdot;f)$ is coercive,~$(u_{t_k})$ is bounded in~$\X$. This implies the existence of a weakly$^*$ convergent subsequence (denoted with the same indices) with limit~$u$. Again, by Assumption~\ref{assmpt:w*-w-c}, this implies that~$\fwd u_{t_k}$ weakly converges to~$\fwd u$ in~$\H$. Due to \eref{ineq:estimates},~$u$ is an element of~$\calN(J)$. Consequently, we can calculate, using weak lower semi-continuity of the norm in~$\H$ and~\eref{ineq:estimates}:
$$\norm{\fwd u-f}_\H\leq\liminf_{k\to\infty}\norm{\fwd u_{t_k}-f}_\H\leq\norm{\fwd u_\infty-f}_\H.$$
Since~$u_\infty$ is the unique minimizer of~\eref{eq:long_time_prob}, this implies that~$u=u_\infty$. The same argument holds true for all cluster points of~$(u_{t_k})$ which shows convergence of the whole sequence.
\end{proof}

In order to obtain a \emph{finite} extinction time, one has to demand the Poincar\'{e}-type inequality of Assumption~\ref{assmpt:poincare} and~$\expr=1$. We define~$E^\expd_t(\cdot;f)\defi E^{\expd,1}_t(\cdot;f)$.

\begin{thm} \label{thm:extinction}
Let~$\expr=1$. Under Assumption~\ref{assmpt:poincare} it holds that
$$ S(f)\defi \sup_{\substack{u \in \calN(J)^{\perp,\fwd}\\ J(u)=1}} \langle f, \fwd u \rangle<\infty$$
and for~$t\geq t_{**}$, given by
\begin{equation}\label{eq:t**}
t_{**} \defi \frac{S(f)}{\Vert f -\fwd \calP^\fwd(f)\Vert^{2-\expd}_\H}
\end{equation}  
if~$f\neq AP^\fwd(f)$ and~$t_{**}=0$ else, it holds that $u_t = \calP^\fwd(f)$ is a  minimizer of~$E^\expd_t(\cdot;f)$. Moreover, for $t > t_{**}$ this is the unique minimizer. Conversely, if~$\calP^\fwd(f)$ is a minimizer of~$E_t^\expd(\cdot;f)$, then~$t\geq t_{**}$.
\end{thm}
\begin{proof}
Using Assumption~\ref{assmpt:poincare} and~$\calP^\fwd(u)=0$ for~$u\in\calN(J)^{\perp,\fwd}$ we calculate
$$S(f)\leq\norm{f}_\H\sup_{\substack{u \in \calN(J)^{\perp,\fwd}\\ J(u)=1}}\norm{\fwd u}_\H\leq C\norm{f}_\H \norm{\fwd} \sup_{\substack{J(u)=1}}J(u)=C\norm{f}_\H \norm{\fwd}<\infty.$$
Now let 
$$	
p_t \defi \frac{\fwd^*(f - \fwd\calP^\fwd(f))}{t \Vert \calP^\fwd(f) - f\Vert^{2-\expd}_\H}.~$$
Then  for any~$u \in  \calN(J)$ we have
$\langle p_t, u \rangle = 0 = J(u)$ which holds in particular for~$u=\calP^\fwd(f)$.  For arbitrary~$u \in \X$ with~$J(u) \neq 0$ we have 
\begin{eqnarray*}
 \langle p_t, u \rangle &=& \frac{\langle f - \fwd\calP^\fwd(f), \fwd u \rangle}{t \Vert f - \fwd\calP^\fwd(f)\Vert^{2-\expd}_\H}
 = \frac{\langle f , \fwd u - \fwd\calP^\fwd(\fwd u) \rangle}{t \Vert  f - \fwd\calP^\fwd(f)\Vert^{2-\expd}_\H} \\ 
&=& \frac{J(u-\calP^\fwd(u)) }{t \Vert  f - \fwd\calP^\fwd(f)\Vert^{2-\expd}_\H} \left\langle f, \fwd\frac{u-\calP^\fwd(u)}{J(u-\calP^\fwd(u))} \right\rangle \\
&\leq& \frac{J(u-\calP^\fwd(u)) }{t \Vert f - \fwd\calP^\fwd(f)\Vert^{2-\expd}_\H} S(f) \leq J(u-\calP^\fwd(u)) = J(u),
\end{eqnarray*} 
using~$t\geq t_{**}$ and~\eref{eq:kernel_J} as well as self-adjointness of $\calP^\fwd$ (cf.~Prop.~\ref{prop:properties_proj}). Thus,~$p_t \in \partial J(\calP^\fwd(f))$ and the optimality condition~\eref{eq:opt_cond_u} is satisfied for~$u_t = \calP^\fwd(f)$. 

Assume that there exists another minimizer~$u$ for~$t > t_{**}$. Then
$$ \frac{1}{\expd}\Vert \fwd u  - f \Vert^\expd_\H + t_{**} J(u) < \frac{1}{\expd}\Vert\fwd u  - f \Vert^\expd_\H + t J(u) \leq \frac{1}{\expd}\Vert \fwd\calP^\fwd(f) - f \Vert^\expd_\H$$
which contradicts the minimization property of~$\calP^\fwd(f)$ for~$E^\expd_{t_{**}}(\cdot;f)$. Let us now assume that~$\calP^\fwd(f)$ is a minimizer. In this case, the optimality condition implies that 
$$p_t\defi\frac{\fwd^*(f-\fwd\calP^\fwd(f))}{t\norm{f-\fwd\calP^\fwd(f)}^{2-\expd}_\H}\in\partial J(\calP^\fwd(f)).$$
Hence, using~$\langle p_t,u\rangle\leq J(u)$ for all~$u\in\X$, we can estimate
\begin{align*}
t_{**}&=\frac{\sup_{u \in \calN(J)^{\perp,\fwd}, J(u)=1}\langle f,\fwd u\rangle}{\norm{f-\fwd\calP^\fwd(f)}^{2-\expd}_\H}=\frac{\sup_{u \in \calN(J)^{\perp,\fwd}, J(u)=1}\langle f-\fwd\calP^\fwd(f),\fwd u\rangle}{\norm{f-\fwd\calP^\fwd(f)}^{2-\expd}_\H}\\
&=t\sup_{\substack{u \in \calN(J)^{\perp,\fwd}\\ J(u)=1}}\langle p_t,u\rangle\leq t
\end{align*} 
which yields the assertion.
\end{proof}

\begin{example}
If~$\X=\H=\R^n$ equipped with the Euclidean inner product,~$\fwd=\id$, and~$J$ is an arbitrary norm on~$\H$, one obtains~$\calP^\fwd=\calP=0$ and, thus, Assumption~\ref{assmpt:poincare} always holds true due to the equivalence of norms on finite dimensional vector spaces. 

If~$\H=\X=L^2(\Omega)$,~$\fwd=\id$,~$J$ is the total variation extended with infinity on~$L^2(\Omega)\setminus \bv(\Omega)$, and~$\Omega\subset\R^n$, Assumption~\ref{assmpt:poincare} is just the Poincar\'{e} inequality for~$\bv$-functions. Here~$\calP^\fwd(u)=\frac{1}{|\Omega|}\int_\Omega u\,\d x$ is the mean value of~$u$ over~$\Omega$.
\end{example}


Summing up the results of the last two sections, the critical time~$t_*>0$ can exist only if~$\expd=1$ whereas~$t_{**}<\infty$ requires~$\expr=1$. In more generality, one can easily extend these results to models of the type $\Phi(\norm{\fwd u-f}_\H)+t\Psi(J(u))$ with convex and differentiable functions~$\Phi$ and~$\Psi$. In this case, the critical times can appear only if~$\Phi'(0)$ or~$\Psi'(0)$, respectively, are positive.

\subsection{Uniqueness of the forward solution path \minrev{for $\expd>1$ or $\expr>1$}}
Let us now prove that for each time $t>0$ the \emph{forward solution path}~$t\mapsto\fwd u_t$ is uniquely determined if $\expd>1$ or $\expr>1$. This is a necessary property for studying finer regularity. Not surprisingly, this follows from the uniqueness of the residuals. 

\begin{thm}[Uniqueness of the forward solution path I]\label{thm:uniqueness_fwd_sol_path}
Let $\expd>1$ or $\expr>1$. Then the set $\lbrace\fwd u_t\st u_t\in\argmin\;E_t^{\expd,\expr}(\cdot;f)\rbrace$ is a singleton for $t>0$. 
\end{thm}
\begin{proof}
Let us first consider the case $t_*>0$. Then necessarily $\expd=1$ and $\expr>1$ holds and by Theorem~\ref{thm:uniqueness_res} we infer that every minimizer of $E_{t}(\cdot;f)$ for $0<t\leq t_*$ has the same residual. Since there is a minimizer with zero residual this has to holds for all minimizers, as well, and this implies that the forward solution path for $0<t\leq t_*$ coincides with the set~$\{f\}$.

Let us now turn to the case $t>t_*$. We use the optimality condition \eref{eq:opt_cond_u} for two minimizers $u_0,u_1$ with $\fwd u_0,\fwd u_1\neq f$ to obtain
\begin{align*}
0=\fwd^*\frac{\fwd u_i-f}{\norm{\fwd u_i-f}^{2-\expd}_\H}+tJ(u_i)^{\expr-1}p_i,
\end{align*}
where $p_i\in\partial J(u_i)$ and $i=0,1$. Subtracting these equalities yields
\begin{align*}
0=\fwd^*\frac{\fwd u_1-f}{\norm{\fwd u_1-f}^{2-\expd}_\H}-\fwd^*\frac{\fwd u_0-f}{\norm{\fwd u_0-f}^{2-\expd}_\H}+t(J(u_1)^{\expr-1}p_1-J(u_0)^{\expr-1}p_0).
\end{align*}
By Theorem~\eref{thm:uniqueness_res} we know that both the residuals and the values of the regularizer are unique and, hence, we can use the maps $R$ and $J$ from \eref{eq:residuum} and \eref{eq:regulariser} to write
\begin{align*}
0=\fwd^*\frac{\fwd u_1-f}{R(t)^{2-\expd}}-\fwd^*\frac{\fwd u_0-f}{R(t)^{2-\expd}}+tJ(t)^{\expr-1}(p_1-p_0).
\end{align*}
Multiplying with $R(t)^{2-\expd}$, taking a duality product with $u_1-u_0$ and using the non-negativity of the symmetric Bregman distance, we infer
\begin{align*}
\left\langle{\fwd u_1-f}-(\fwd u_0-f),\fwd u_1-\fwd u_0\right\rangle\leq0.
\end{align*}
which is equivalent to $\norm{\fwd u_1-\fwd u_0}^2_\H\leq0$ and shows $\fwd u_0=\fwd u_1$. 
\end{proof}

It remains to study what happens for $\expd=\expr=1$. Since in this case both the data fidelity and the regularizing term of the energy functional \eref{eq:functional} are not strictly convex, one cannot expect uniqueness of the forward solution path for parameters $t\in[t_*,t_{**}]$. However, for values of $t$ where non-uniqueness occurs, we are able to confine the set of possible forward solutions to a one-parameter family.

\begin{thm}[Uniqueness of the forward solution path II]\label{thm:non_uniquness_fwd}
Let $t\geq t_*$. Then it holds
\begin{align}\label{incl:one-param_fam}
\{\fwd u\st u\in\argmin\;E_{t}^{1,1}(\cdot;f)\}\subset\{f+c(\fwd \hat{u}-f)\st c\geq0\},
\end{align}
where $\hat{u}$ is an arbitrary minimizer of $E_t^{1,1}(\cdot;f)$ fulfilling $\fwd\hat{u}\neq f$.
\end{thm}
\begin{proof}
The only non-trivial case is $\fwd u\neq f$ since otherwise $c=0$ can be chosen in \eref{incl:one-param_fam}. As before, we obtain by subtracting the optimality conditions \eref{eq:opt_cond_u} of $u$ and $\hat{u}$ that
\begin{align}\label{eq:opt_diff}
0=\fwd^*\frac{\fwd u-f}{\norm{\fwd u-f}_\H}-\fwd^*\frac{\fwd \hat{u}-f}{\norm{\fwd \hat{u}-f}_\H}+t(p-\hat{p}),
\end{align}
where $p$ and $\hat{p}$ denote the corresponding subgradients.
We shortcut $w\defi\fwd u-f$ and $\hat{w}\defi\fwd\hat{u}-f$, multiply with $u-\hat{u}$, and use the non-negativity of the symmetric Bregman distance to obtain
\begin{align}
0\geq\left\langle\frac{w}{\norm{w}_\H}-\frac{\hat{w}}{\norm{\hat{w}}_\H},w-\hat{w}\right\rangle=\norm{w}_\H+\norm{\hat{w}}_\H-\langle w,\hat{w}\rangle\frac{\norm{w}_\H+\norm{\hat{w}}_\H}{\norm{w}_\H\norm{\hat{w}}_\H}\geq0,
\end{align}
where the second inequality follows from Cauchy-Schwarz. This immediately implies $\langle w,\hat{w}\rangle=\norm{w}_\H\norm{\hat{w}}_\H$ which is only possible if $w=c\hat{w}$ with $c\geq 0$. Hence, we obtain 
\begin{align}\label{eq:aff_res}
\fwd u-f=c(\fwd \hat{u}-f)
\end{align}
which is equivalent to $\fwd u=f+c(\fwd\hat{u}-f)$. This closes the proof.
\end{proof}

\begin{rem}
Note that in case $c\neq 1$, which corresponds to non-uniqueness of the forward solution, \eref{eq:aff_res} can be rewritten as
\begin{align}\label{eq:exact_recon}
\fwd\left(\frac{u-c\hat{u}}{1-c}\right)=f,
\end{align}
which means that -- in case of non-uniqueness -- one can construct an element from the two minimizers which fulfills the range condition \eref{eq:range_cond}. This is a counter-intuitive behavior since one would not expect the two \emph{regularized} solutions to carry sufficiently much information to allow for the exact reconstruction of the datum~$f$. Indeed, if $f\notin\ran\fwd$ -- which can be interpreted as \emph{noisy data} -- equation~\eref{eq:exact_recon} is a contradiction and, hence, the forward solution path is unique in this case. 
\end{rem}

Despite the considerations of the previous remark, on cannot expect uniqueness of the forward solution path, in general. This will be illustrated in the following example.

\begin{example}
Let $X=\H=\R^2$, $\fwd=\begin{pmatrix}
2 & 1 \\ 
1 & 1
\end{pmatrix}$, $f=\begin{pmatrix}
-2,  -3
\end{pmatrix}^T$ and $J(u)=\norm{u}_1=|u_1|+|u_2|$. Then the forward solution path is not unique in $t_*={1}/{\sqrt{13}}$ and in $t=1$. This can be seen as follows: It is well-known that the subdifferential of the 1-norm is given by the multivalued signum function, i.e, for $u\in\R^n$ it holds component-wise $\partial\norm{u}_1=\sgn(u)$, where $\sgn(\cdot)$ denotes the multi-valued sign function. In addition, since $\fwd$ is invertible, the vector $\dataX\defi\begin{pmatrix}
1,  -4
\end{pmatrix}^T$
is the unique vector to fulfill $\fwd\dataX=f$. Hence, $\partial \norm{\dataX}_1=p=\begin{pmatrix}
1,-1
\end{pmatrix}^T$ and $q=(\fwd^*)^{-1}p=\fwd^{-1}p=\begin{pmatrix}
2,-3
\end{pmatrix}^T$ is the unique source element. This implies that $t_*=1/\norm{q}_2=1/\sqrt{13}>0$. It can be easily checked using the optimality condition \eref{eq:opt_cond_u} that all members of the family 
$$u_\lambda\defi\begin{pmatrix}
1 \\ 
-4
\end{pmatrix}-\lambda\begin{pmatrix}
5 \\ 
-8
\end{pmatrix},\quad\lambda\in\left[0,\frac{1}{5}\right],$$
are minimizers for $t=t_*$ and, similarly, that all members of
 $$u_\lambda\defi\begin{pmatrix}
1 \\ 
-4
\end{pmatrix}-\lambda\begin{pmatrix}
1 \\ 
-2
\end{pmatrix},\quad\lambda\in\left[1,2\right],$$
are minimizers for $t=1$. Hence, due to the invertibility of $\fwd$, also the corresponding forward solution paths are not unique. The strategy to find such non-unique solutions is using the ansatz $\fwd u_\lambda=f-\lambda q$ (cf.~\eref{eq:aff_res}), where $\fwd^*q$ is a subgradient of $u_\lambda$ for $\lambda$ in a suitable interval. 

Furthermore, since $\fwd$ is invertible, we can use the change of variables $v=\fwd u$ to obtain
$$\min_{u\in\R^2}\norm{\fwd u-f}_2+t\norm{u}_1=\min_{v\in\R^2}\norm{v-f}+t\norm{\fwd^{-1}v}_1.$$
Hence, we can have non-uniqueness even if the forward operator is trivial, i.e., equals the identity. 
\end{example}

An important consequence of the uniqueness of the forward solution is the continuity of the residual map~$t\mapsto R(t)$.

\begin{cor}[Continuity of the residuals]\label{cor:cont_res}
Let~$\expd>1$ or $\expr> 1$. Then the map~$t\mapsto R(t)$ is continuous for all~$t>0$.
\end{cor}
\begin{proof}
The continuity follows from a straightforward generalization of the proof of Claim~3 in~\cite{chan2005aspects}, using that~$\fwd^*$ is weak$^*$-to-weak continuous and the Hilbert norm is weak lower semi-continuous.
\end{proof}

From the uniqueness of the forward solution path and the residuals we immediately obtain

\begin{cor}\label{cor:unique_sol_path}
Under the conditions of Theorem~\ref{thm:uniqueness_fwd_sol_path} it holds:
\begin{enumerate}
\item For every~$t>t_*$ the subgradient~$p_t$ from the optimality conditions~\eref{eq:opt_conf_f} and~\eref{eq:opt_cond_u} is uniquely determined.
\item If~$\fwd$ is injective, then uniqueness of the forward solution path implies uniqueness of the solution path~$t\mapsto u_t$.  
\end{enumerate}
\end{cor}

\subsection{\LB{Relation of the problems}}\label{sec:relation}

In this section, we will deal with the mutual relation of minimizers of~$E_t^{\expd,\expr}(\cdot;f)$ for different values of~$\expd$ and~$\expr$. The structure of the subgradient \eref{eq:opt_cond_u} suggests that as long as $\norm{\fwd u_t-f}_\H,J(u_t)\neq0$, one can switch back and forth between minimizers corresponding to different choices of the exponents $\expd,\expr$ by adapting the regularization parameter $t$. Foreshadowing, one has one-to-one correspondences of all minimizers within the critical parameter range~$(t_*,t_{**})$ where~$t_*$ and~$t_{**}$ can attain the values~$0$ or~$\infty$, respectively. For instance, minimizers of~$E_t^{1,2}(\cdot;f)$ for~$t\in(t_*,\infty)$ correspond exactly to those of~$E^{2,1}_\tau(\cdot;f)$ for~$\tau\in(0,\tau_{**})$. Exemplary, we will prove this equivalence for minimizers of~$E_t^{\expd,1}(\cdot;f)$ with~$\expd\geq1$ and~$E_\tau^{2,1}(\cdot;f)$, the latter being the ``standard'' variational problem with squared norm and one-homogeneous regularization. Since both models possess finite extinction time due to~$\expr=1$, we will obtain full equivalence for~$t\in(t_*,\infty)$ and~$\tau\in(0,\infty)$. Note that in the following, the expression $u_t$ will correspond to minimizers of $E_t^{\expd,1}(\cdot;f)$ whereas $v_\tau$ will only be used for minimizers of $E^{2,1}_\tau(\cdot;f)$. In particular, $R(t)=\norm{\fwd u_t-f}_\H$ and $R(\tau)=\norm{\fwd v_\tau-f}_\H$ denote the respective residuals and are not to be confused. Furthermore, we remind of the fact that the residual $\norm{\fwd u_t-f}_\H$ is not uniquely determined if $\expd=1$. By the optimality condition \eref{eq:opt_cond_u} we obtain the following two lemmas.

\begin{lemma}\label{lem:u_t-to-v_tau}
Let~$t>t_*$ and~$u_t$ be a minimizer of~$E^{\expd,1}_t(\cdot;f)$. Then~$u_t$ is a minimzer of~$E^{2,1}_\tau(\cdot;f)$ with~$\tau\defi t R(t)^{2-\expd}$.
\end{lemma}

\begin{lemma}\label{lem:v_tau-to-u_t}
Let~$\tau > 0$ and~$v_\tau$ be the minimizer of~$E_\tau^{2,1}(\cdot;f)$. Then~$v_\tau$ is also a minimizer of~$E^{\expd,1}_t(\cdot;f)$ with~$t\defi T(\tau)\defi\tau R(\tau)^{\expd-2}$.
\end{lemma}

\begin{thm}\label{thm:relation}
The map~$T:(0,\infty)\rightarrow(0,\infty),\;\tau\mapsto T(\tau)\defi \tau R(\tau)^{\expd-2}$ is well-defined, non-decreasing, and surjective. If $\expd>1$, it is even a bijection with continuous inverse~$S(t)\defi t R(t)^{2-\expd}$. 
\end{thm} 
\begin{proof}
Since by Theorem~\ref{thm:uniqueness_res} the residuals of minimizers with strictly convex data term are unique, map~$T$ is well-defined. By Corollary~\ref{cor:cont_res} it follows that $T$ is continuous. Let us first consider the case $\expd>1$. Then similarly $S$ is well-defined and continuous. Furthermore, it follows from the uniqueness of the residuals that $S$ and $T$ are mutual inverses. 

Finally, $T$ is non-decreasing which can be seen as follows. For~$\expd\geq2$ this is obvious as~$T$ is the product of non-decreasing functions (cf.~Lemma~\ref{lem:monot_res}). For~$\expd\in(1,2)$ the same holds true for~$S$. Since they are inverses, shows that both~$T$ and~$S$ are increasing for arbitrary $\expd>1$.

Let us now address the case $\expd=1$. As we have seen, the residuals $R(t)$ are not unique in general and therefore, the map $S(t)\defi tR(t)$ is not well-defined. However, by Lemmas~\ref{lem:u_t-to-v_tau} and \ref{lem:v_tau-to-u_t} we infer that $T$ is still surjective. Furthermore, being the pointwise limit of the increasing functions $\tau R(\tau)^{\expd-2}$ for $\alpha\searrow 1$ shows that $T(\tau)=\tau/R(\tau)$ is non-decreasing.  
\end{proof}

\begin{rem}[Bayesian interpretation]
The relation of the problems for different values of~$\expd$ and~$\expr$ can also be interpreted in terms of Bayesian models for inverse problems (cf.~\cite{stuart2010inverse}). Under appropriate conditions,~$E_t^{\expd,\expr}(\cdot;f)$ can be interpreted as the Onsager-Machlup functional of a posterior distribution and its minimizer is the maximum a-posteriori probability (MAP) estimate (cf.~\cite{helin2015maximum,agapiou2018sparsity}). In the finite-dimensional case the posterior density is often simply modeled as~$p(u|f) \sim \exp(-c E_t^{\expd,\expr}(u;f))$. In practice,~$\expd$ is determined from the noise modelling, while one usually chooses~$\expr=1$ based on the standard formulation of the variational problem. Essentially, the posterior distribution is extrapolated from the collection of MAP estimates, in practice. However, the equivalence of the minimization problems for different~$\expr$ shows that there is a variety of posterior distributions leading to the same MAP estimates for any~$f$. The behaviour of the posterior however can differ strongly, in particular in degenerate cases such as~$\bv$ (cf.~\cite{comelli2011novel,lucka2012fast,burger2014maximum}).
\end{rem}

\subsection{\minrev{Uniqueness of the forward solution path for $\expd=\expr=1$}}

\minrev{The results of the previous section allow us to characterize (non-)uniqueness of the forward solution path also in the degenerate case $\expd,\expr=1$.}
\begin{thm}\label{thm:non-unique-affine}
Non uniqueness of the forward solution of $E_t^{1,1}(\cdot;f)$ in some $t>0$ is in one-to-one correspondence to an affine forward solution path of the form $f-\tau q$ of $E^{2,1}_\tau(\cdot;f)$ for $\tau\in[\tau_0,\tau_1]$ where $0\leq\tau_0<\tau_1$. 
\end{thm}
\begin{proof}
Let us first assume that $f-\tau q$ is the forward solution path of $E_\tau^{2,1}(\cdot;f)$ for $\tau\in[\tau_0,\tau_1]$ and $0\leq\tau_0<\tau_1$. Then $q\neq 0$ holds since $f$ cannot be a minimizer for any positive value of $\tau$. Hence the time reparametrization $T$ reduces to $T(\tau)={\tau}/{R(\tau)}=1/{\norm{q}_\H},$ which means by Lemma~\ref{lem:v_tau-to-u_t} that $f-\tau q$ is also the forward solution of $E_t^{1,1}(\cdot;f)$ for $t\defi{1}/{\norm{q}_\H}$. Since $\tau$ runs in a proper interval this implies the non-uniqueness of the forward solution in~$t$.

Conversely, let us assume that the forward solution of $E_t^{1,1}(\cdot;f)$ is not unique. Then there exist $u_0,u_1\in\argmin\;E_t^{1,1}(\cdot;f)$ such that $\fwd u_0\neq\fwd u_1$. Due to convexity also the convex combinations $u_\lambda\defi(1-\lambda)u_0+\lambda u_1$ for $\lambda\in[0,1]$ are minimizers of $E_t^{1,1}(\cdot;f)$ and it holds
\begin{align}\label{eq:convex_residuals}
\fwd u_\lambda-f=(1-\lambda)(\fwd u_0-f)+\lambda(\fwd u_1-f).
\end{align}
We distinguish two cases. If $\fwd u_0=f$ and $\fwd u_1\neq f$ (this corresponds to $t=t_*$) we have
\begin{align}\label{eq:convex_res_case1}
\fwd u_\lambda-f=\lambda(\fwd u_1-f).
\end{align}
If, however, $\fwd u_0,\fwd u_1\neq f$ we can use \eref{eq:aff_res} to write $\fwd u_1-f=c(\fwd u_0-f)$ for some $c\in(0,\infty)\setminus\{1\}$ which, together with \eref{eq:convex_residuals}, implies
\begin{align}\label{eq:convex_res_case2}
\fwd u_\lambda-f=(1+\lambda(c-1))(\fwd u_0-f).
\end{align}
In any case, we can define numbers $\tau_0\defi t\norm{\fwd u_0-f}\geq0$ and $\tau_1\defi t\norm{\fwd u_1-f}$. In the first case we have $0=\tau_0<\tau_1$ and in the second case -- after possibly exchanging the roles of $u_0$ and $u_1$ -- we can assume $c>1$ such that $0<\tau_0<\tau_1$ holds. Next, we observe that, due to Lemma~\ref{lem:u_t-to-v_tau}, $u_\lambda$ is a minimizer of $E_\tau^{2,1}(\cdot;f)$ with $\tau\defi t\norm{\fwd u_\lambda-f}_\H$. By using \eref{eq:convex_res_case1} or \eref{eq:convex_res_case2}, respectively, we infer that in both cases it holds $\tau=(1-\lambda)\tau_0+\lambda\tau_1\in[\tau_0,\tau_1]$
which is equivalent to $\lambda=(\tau-\tau_0)/(\tau_1-\tau_0)\in[0,1]$. Plugging this expression for $\lambda$ into \eref{eq:convex_res_case1} or \eref{eq:convex_res_case2}, respectively, the forward solution of $E_\tau^{2,1}(\cdot;f)$ in $\tau\in[\tau_0,\tau_1]$ is given by $f-\tau\left[(f-\fwd u_1)/\tau_1\right]$ in both cases, which follows after some algebra and allows us to conclude.
\end{proof}

\begin{cor}\label{cor:aae-uniqueness}
The forward solution of $E_t^{1,1}(\cdot;f)$ is uniquely determined for almost every $t>0$.
\end{cor}
\begin{proof}
According to Theorem~\ref{thm:non-unique-affine}, non-uniqueness implies a forward solution path of the form $f-\tau q$ of the problem $E_\tau^{2,1}(\cdot;f)$ for $\tau\in[\tau_0,\tau_1]$ which implies that $T(\tau)=t$ is constant for $\tau\in[\tau_0,\tau_1]$. Since the set $\{t\in\R\st|\{\tau:T(\tau)=t\}|>0\}$ has Lebesgue measure zero, we can conclude. 
\end{proof}

\section{Regularity of the forward solution path}\label{sec:regularity}
In this section, we investigate regularity of the {forward solution path}~$t\mapsto\{\fwd u_t\st u_t\in\argmin\;E^{\expd,\expr}_t(\cdot;f)\}$ which we have shown to be a single-valued map for~$\expd>1$ or $\expr>1$ in the previous section. As already mentioned, when using the minimization of~\eref{eq:functional} for obtaining nonlinear spectral decompositions of the data~$f$, one typically computes derivatives of the (forward) solution path with respect to~$t$. While these solution paths can be shown to be sufficiently regular under some finite dimensional assumptions (cf.~the discussion in \sref{sec:affine}), a general study of their regularity in a Banach or Hilbert space setting is still pending. Our results are a first contribution in this direction and the topic will remain subject to future research.

\subsection{\minrev{Residual bounds under range or source condition}}
\label{sec:res_bounds}

In this section we prove growth rates of the residual $R(t)$ which will be used to improve the subsequent Lipschitz estimates of the forward solution path close to zero under range or source condition. This can also be interpreted as H\"older continuity of the forward solution path close to zero in the case~$\expd>1$. First, we state a preparatory lemma.

\begin{lemma}\label{lem:bounded_vt}
Let~$\expd,\expr\geq1$, and~$u_t$ be a minimizer of~$E_t^{\expd,\expr}(\cdot;f)$ for $t>t_*$. If~\eref{eq:range_cond} and~\eref{eq:source_cond} hold, then~$q_t$, defined as
\begin{align}\label{eq:source_element}
q_t\defi\frac{f-\fwd u_t}{tR(t)^{2-\expd}J(t)^{\expr-1}},
\end{align}
fulfills $\norm{q_t}_\H\leq\min\left\lbrace s_*,{R(t)^{\expd-1}}/{(t J(t)^{\expr-1})}\right\rbrace.$
\end{lemma}
\begin{proof}
By the optimality conditions~\eref{eq:opt_cond_u} we infer that~$p_t\defi\fwd^*q_t\in\partial J(u_t)$. Furthermore, letting~$\hat{q}\in\H$ and $\hat{u}$ be such that~$p_0\defi\fwd^*\hat{q}\in\partial J(\hat{u})$ and~$\norm{\hat{q}}_\H=s_*$, we calculate
$$\langle q_t-\hat{q},q_t\rangle=-\frac{1}{tR(t)^{2-\expd}J(t)^{\expr-1}}\langle q_t-\hat{q},\fwd u_t-f\rangle=-\frac{1}{tR(t)^{2-\expd}J(t)^{\expr-1}}\langle p_t-p_0,u_t-\dataX\rangle\leq0,$$
which is equivalent to
$$\norm{q_t}_\H^2\leq\langle\hat{q},q_t\rangle.$$
With Cauchy-Schwarz this implies~$\norm{q_t}_\H\leq\norm{\hat{q}}_\H=s_*$. The other upper bound is trivial.
\end{proof}

Now we are ready to prove the growth bounds of the residuals. Note that the growth in zero can be estimated more sharply when demanding the source condition~\eref{eq:source_cond}.
\begin{lemma}\label{lem:hoelder}
Let~\eref{eq:range_cond} hold true. It holds for all~$t>t_*$
\begin{equation}\label{ineq:hoelder}
R(t)\leq t^\frac{1}{\expd}J_{\min}^\frac{\expr}{\expd}.
\end{equation}  
Under condition~\eref{eq:source_cond} and if $\expd>1$ it even holds 
\begin{equation}\label{ineq:hoelder_w_source}
R(t)\leq t^\frac{1}{\expd-1}s_*^\frac{1}{\expd-1}J_{\min}^\frac{\expr-1}{\expd-1}.
\end{equation}
\end{lemma}

\begin{proof}
From the optimality condition \eref{eq:opt_cond_u} for~$u_t$ we obtain
\begin{align*}
0&=R(t)^{\expd-2}\fwd^*(\fwd u_t-f)+tJ(t)^{\expr-1} p_t,
\end{align*}
where~$p_t\in\partial J(u_t)$. Reordering yields~$\fwd^*\fwd(u_t-\dataX)=-tR(t)^{2-\expd}{J(t)^{\expr-1}}~p_t$ and by taking the duality product with~$u_t-\dataX$ we obtain
\begin{align*}
R(t)^2=\norm{\fwd u_t-f}_\H^2=&tR(t)^{2-\expd}J(t)^{\expr-1}\langle p_t,\dataX-u_t\rangle\\
\leq &tR(t)^{2-\expd}J(t)^{\expr-1}(J_{\min}-J(u_t))\\
\leq &{t}{R(t)^{2-\expd}} J_{\min}^\expr,
\end{align*}
where we used that $J$ is decreasing (cf.~Lemma~\ref{lem:monot_res}) and $J(u_t)\geq0$. Given~\eref{eq:source_cond}, we define~$q_t$ as in~\eref{eq:source_element} and use Lemma~\ref{lem:bounded_vt} to write 
$$\langle p_t,\dataX-u_t\rangle=\langle\fwd^*q_t,\dataX-u_t\rangle\leq\langle q_t,f-\fwd u_t\rangle\leq\norm{q_t}_\H R(t)\leq s_* R(t)$$
which can be used to obtain the second inequality if $\expd>1$.
\end{proof}

\subsection{Lipschitz continuity of the forward solution path for $\expd > 1$ or $\expr>1$}
In this section we address the Lipschitz continuity of the forward solution path in the case that it is uniquely determined. It will turn out to be Lipschitz continuous in the range of positive parameters $t$. For $t\searrow 0$ the general estimates will break down which is obvious since the solution instantaneously changes from the noisy data to being regularized as $t$ gets positive. However, if the range or source conditions hold, the rate of change can be slightly tamed \minrev{by employing the results of Section~\ref{sec:res_bounds}.}

The following lemma is the basis for our regularity estimates.
\begin{lemma}
Let~$\expd,\expr\geq1$. For~$t_*<s<t$ the estimate
\begin{equation}\label{ineq:conti}
\Vert \fwd u_t-\fwd u_s \Vert_\H \leq \frac{tJ(t)^{\expr-1}R(t)^{2-\expd}-sJ(s)^{\expr-1}R(s)^{2-\expd}}{tJ(t)^{\expr-1}}R(t)^{\expd-1}
\end{equation}
holds, where~$u_t$ and~$u_s$ are minimizers of~$E^{\expd,\expr}_t(\cdot;f)$ and~$E^{\expd,\expr}_s(\cdot;f)$, respectively.
\end{lemma}
\begin{proof}
Defining~$\tilde{t}\defi tJ(t)^{\expr-1}R(t)^{2-\expd}$ and~$\tilde{s}$ analogously, we obtain from the optimality conditions for $p_t$ and~$p_s$ given by \eref{eq:opt_cond_u}:
$$ \frac{1}{\tilde{s}} \fwd^*\fwd(u_t - u_s) + p_t - p_s = \frac{\tilde{t}-\tilde{s}}{\tilde{s}\tilde{t}} \fwd^*(\fwd u_t -f).$$
Taking a duality product with~$u_t -u_s$ and using non-negativity of the symmetric Bregman distance yields
\begin{align*}
\frac{1}{\tilde{s}}\norm{\fwd u_t-\fwd u_s}_\H^2\leq \frac{\tilde{t}-\tilde{s}}{\tilde{s}\tilde{t}}\langle\fwd u_t-f,\fwd u_t-\fwd u_s\rangle.
\end{align*}
Application of the Cauchy-Schwarz inequality to the right hand side and simple reordering leads to
\begin{align*}
\Vert \fwd u_t-\fwd u_s \Vert_\H \leq \frac{\tilde{t}-\tilde{s}}{\tilde{t}}R(t)=\frac{\tilde{t}-\tilde{s}}{tJ(t)^{\expr-1}}R(t)^{\expd-1}.
\end{align*}
Plugging in the definitions of $\tilde{t}$ and $\tilde{s}$ concludes the proof.
\end{proof}

This result also included the case $\expd=\expr=1$. Since, however, in this case the forward solution path is not even uniquely defined, one cannot expect continuity properties. Thus, the following statements will always require that one of the weights $\expd$ and $\expr$ is larger than one. In particular, the maps $R(t)$ and $J(t)$ are well-defined in that case.

\begin{cor}[Continuity of the forward solution path]
If $\expd>1$ or $\expr>1$, estimates~\eref{ineq:conti} together with the continuity of the residuals and the regularizers (cf.~Corollary~\ref{cor:cont_res}) shows that the forward solution path~$t\mapsto\fwd u_t$ is continuous for all~$t>t_*$. Continuity in $t=t_*$ follows from the continuity of the residuals. 
\end{cor}

For~$\expd\geq2$ one can directly obtain Lipschitz estimates of the forward solution path. As already mentioned, the estimates close to zero can be improved by assuming the source condition. 

\begin{lemma}\label{lem:lipschitz}
Let~$\expd\geq 2$ and~$\expr\geq1$. \minrev{Let~$0<s<t$ and let $u_t$ and $u_s$ be minimizers of~$E^{\expd,\expr}_t(\cdot;f)$ and~$E^{\expd,\expr}_s(\cdot;f)$, respectively. Then the estimate
\begin{equation}\label{ineq:lipschitz}
\norm{\fwd u_t-\fwd u_s}_\H\leq\frac{t-s}{t}R(t)\leq\frac{t-s}{t}\norm{f}_\H
\end{equation}
holds. This estimate can be improved to
\begin{align}\label{ineq:lipschitz_w_range}
\norm{\fwd u_t-\fwd u_s}_\H&\leq C_R \frac{t-s}{t^\frac{\expd-1}{\expd}},\qquad\text{if }\eref{eq:range_cond}\text{ holds},\\
\label{ineq:lipschitz_w_source}
\norm{\fwd u_t-\fwd u_s}_\H&\leq C_S \frac{t-s}{t^\frac{\expd-2}{\expd-1}},\qquad\text{if also }\eref{eq:source_cond}\text{ holds},
\end{align}
respectively}, with constants~$C_R\defi J_{\min}^\frac{\expr}{\expd}$ and~$C_S\defi s_*^\frac{1}{\expd-1}J_{\min}^\frac{\expr-1}{\expd-1}$.
\end{lemma}

\begin{proof}
We start from~\eref{ineq:conti}:
\begin{align*}
\Vert \fwd u_t-\fwd u_s \Vert_\H &\leq 
\frac{tJ(t)^{\expr-1}R(t)^{2-\expd}-sJ(s)^{\expr-1}R(s)^{2-\expd}}{tJ(t)^{\expr-1}}R(t)^{\expd-1}\\
&=\frac{t-s\left(\frac{J(s)}{J(t)}\right)^{\expr-1}\left(\frac{R(s)}{R(t)}\right)^{2-\expd}}{t}R(t).
\end{align*}
If we now use that~$\expd\geq 2$ and that for~$s<t$ it holds~$J(t)\leq J(s)$ and~$R(s)\leq R(t)$, we obtain
$$\norm{\fwd u_t-\fwd u_s}_\H\leq\frac{t-s}{t}R(t)\leq\frac{t-s}{t}\norm{f}_\H.$$
Here we employed the a-priori estimate~$R(t)\leq  \Vert f \Vert_\H$ which follows from~$E^{\expd,\expr}_t(u_t;f) \leq E^{\expd,\expr}_s(0;f)$. Under~\eref{eq:range_cond} or~\eref{eq:source_cond} one uses Lemma~\ref{lem:hoelder} to further estimate~$R(t)$. 
\end{proof}

We obtain the following regularity result for the forward solution path.

\begin{thm}[Lipschitz continuity of the forward solution path I]\label{thm:dv}
Let~$\expd\geq2$,~$\expr\geq1$. The forward solution path~$t\mapsto\fwd u_t$ is Lipschitz continuous on~$(\delta,\infty)$ for all $\delta>0$. Hence,~$(\fwd u_t)'$ exists almost everywhere in~$(\delta,\infty)$ and it holds
\begin{equation}\label{ineq:Av_t'}
\norm{(\fwd u_t)'}_\H\leq\frac{R(t)}{t}\leq \frac{\norm{f}_\H}{t}.
\end{equation}
\minrev{This estimate can be improved to 
\begin{align}\label{ineq:Av_t'_range}
\norm{(\fwd u_t)'}_\H&\leq C_R t^\frac{1-\expd}{\expd},\qquad\text{if }\eref{eq:range_cond}\text{ holds},\\
\label{ineq:Av_t'_source}
\norm{(\fwd u_t)'}_\H&\leq C_S t^\frac{2-\expd}{\expd-1},\qquad\text{if also }\eref{eq:source_cond}\text{ holds},
\end{align}
respectively}, for almost every~$t\in (0,\infty)$. Furthermore, if~$\expd=2$ and assuming conditions~\eref{eq:range_cond} and~\eref{eq:source_cond}, the Lipschitz continuity becomes global on~$[0,\infty)$. 
\end{thm}
\begin{proof}
Lipschitz continuity of~$t\mapsto\fwd u_t$ is a direct consequence of estimate~\eref{ineq:lipschitz}. Since~$\H$, being a Hilbert space, has the Radon-Nikodym property (cf.~\cite{stegall1975radon}, for instance), we can deduce from a generalization of Rademacher's theorem~\cite{aronszajn1976differentiability,kirchheim1994rectifiable} that~$(\fwd u_t)'$ exists almost everywhere on~$(0,\infty)$. Estimates~\eref{ineq:Av_t'},~\eref{ineq:Av_t'_range}, and \eref{ineq:Av_t'_source} are direct consequences of~\eref{ineq:lipschitz},~\eref{ineq:lipschitz_w_range}, and~\eref{ineq:lipschitz_w_source}. Global Lipschitz continuity on the whole real line for~$\expd=2$ and \eref{eq:source_cond} follows from~\eref{ineq:lipschitz_w_source}.
\end{proof}

\begin{cor}\label{cor:R_J_Lipschitz}
Let~$\expd\geq2$ and~$\expr\geq1$. Then the maps~$t\mapsto R(t)$ and~$t\mapsto J(t)$ are Lipschitz continuous on~$(\delta,\infty)$ for all $\delta>0$.
\end{cor}
\begin{proof}
The first assertion is an immediate consequence of the reverse triangle inequality:
$$|R(t)-R(s)|=|\norm{\fwd u_t-f}_\H-\norm{\fwd u_s-f}_\H|\leq\norm{\fwd u_t-\fwd u_s}_\H.$$
Since by estimate~\eref{ineq:lipschitz} the forward solution path is Lipschitz, the same holds for $R$. For the second claim, let~$0<s<t$ and let~$u_s$ and~$u_t$ denote corresponding minimizers. Thus, it holds
$$\frac{\expr}{s\expd}R(s)^\expd+J(s)^\expr\leq\frac{\expr}{s\expd}R(t)^\expd+J(t)^\expr$$
from which we deduce
\begin{align*}
|J(s)^\expr-J(t)^\expr|=J(s)^\expr-J(t)^\expr\leq\frac{\expr}{s\expd}\left(R(t)^\expd-R(s)^\expd\right).
\end{align*}
Since~$R(\cdot)$ is Lipschitz on $(\delta,\infty)$ for all $\delta>0$, the same holds for~$R(\cdot)^\expd$ and for~$t\mapsto J(t)^\expr$. Applying the~$\expr$-th root, preserves local Lipschitz continuity away from zero and hence we can conclude.
\end{proof}

In order to proceed to the case~$1\leq\expd<2$, where $\expd=1$ requires $\expr>1$, we use the relation between the different formulations established in \sref{sec:relation}. For simplicity, we will only consider the case $1<\expd<2$ and $\expr=1$. Defining~$\tau\mapsto T(\tau)=\tau R(\tau)^{\expd-2}$ as in Theorem~\ref{thm:relation}, one observes that, due to Corollary~\ref{cor:R_J_Lipschitz}, function~$T$ is Lipschitz continuous on $(\delta,\infty)$ for all $\delta>0$. Hence, its derivative~$T'$ exists almost everywhere in~$(\delta,\infty)$ and it holds
\begin{align}
T'(\tau)=\frac{\d}{\d\tau}\tau R(\tau)^{\expd-2}=\tau(\expd-2)R(\tau)^{\expd-4}{\langle\fwd v_\tau-f,(\fwd v_\tau)'\rangle}+R(\tau)^{\expd-2}.
\end{align}
Here, we used that also~$R'(\tau)$ exists almost everywhere according to Corollary~\ref{cor:R_J_Lipschitz}, and can be computed with the chain rule: $R'(\tau)={\langle\fwd v_\tau-f,(\fwd v_\tau)'\rangle}/{R(\tau)}$. Thus,~$T'$ is positive if and only if $\tau(2-\expd)R(\tau)^{-2}{\langle\fwd v_\tau-f,(\fwd v_\tau)'\rangle}<1.$ For~$1<\expd<2$, this inequality is true due to Cauchy-Schwarz and estimate~\eref{ineq:Av_t'} which can be used to bound $(\fwd v_\tau)'$. Hence, in that case also~$S$, the inverse of $T$, is a Lipschitz function. Consequently, we obtain Lipschitz continuity for minimizers~$u_t$ of~$E_t^{\expd,1}(\cdot;f)$ with~$1<\expd<2$ since~$\fwd u_t=\fwd v_{S(t)}$ is a composition of Lipschitz functions. By setting $T(\tau)=\tau R(\tau)^{\expd-2}J(\tau)^{1-\expr}$ this argument can easily be repeated for~$\expr>1$, which makes the calculations more cumbersome but leads to the same results. In this case, also $\expd=1$ and $\expr>1$ yields the desired Lipschitz continuity. Hence, the assumption~$\expd\geq2$ in Corollary~\ref{cor:R_J_Lipschitz} and Theorem~\ref{thm:dv} can be relaxed to~$\expd>1$ or $\expr>1$ without loosing Lipschitz continuity or differentiability of the forward solution path. However, estimates~\eref{ineq:Av_t'} and~\eref{ineq:Av_t'_source} need to be adapted. To keep the presentation short, we only formulate the estimates for $\expr=1$.

\begin{thm}[Lipschitz continuity of the forward solution path II]\label{thm:dv_p<1}
Let~$1\leq \expd<2$ and $\expr\geq 1$ such that $\expd>1$ or $\expr>1$. The forward solution path~$t\mapsto\fwd u_t$ is Lipschitz continuous on~$(\delta,\infty)$ for all $\delta>0$. Furthermore,~$(\fwd u_t)'$ exists almost everywhere in~$(\delta,\infty)$ and it holds for almost all~$t\in (0,\infty)$, $1<\expd<2$, and $\expr=1$
\begin{equation}\label{ineq:Av_t'_gen}
\norm{(\fwd u_t)'}_\H\leq\frac{1}{\expd-1}\frac{R(t)}{t}\leq\frac{1}{\expd-1} \frac{\norm{f}_\H}{t}.
\end{equation}
\minrev{This estimate can be improved to
\begin{align}\label{ineq:Av_t'_range_gen}
\norm{(\fwd u_t)'}_\H&\leq \frac{1}{\expd-1}C_R t^\frac{1-\expd}{\expd},\qquad\text{if }\eref{eq:range_cond}\text{ holds},\\
\label{ineq:Av_t'_source_gen}
\norm{(\fwd u_t)'}_\H&\leq \frac{1}{\expd-1}C_S t^\frac{2-\expd}{\expd-1},\qquad\text{if also }\eref{eq:source_cond}\text{ holds},
\end{align}
respectively}, for almost every~$t\in (0,\infty)$. 
\end{thm}

\begin{proof}
For simplicity, we only consider the case~$\expr=1$. It remains to prove the bound~\eref{ineq:Av_t'_gen}. To this end, we let~$u_t$ denote a minimizer of~$E_t^{\expd,1}(\cdot;f)$. Then it holds according to the previous results that~$u_t=v_\tau$ with~$\tau=S(t)$ and with the chain rule together with~\eref{ineq:Av_t'} we obtain
$$\norm{(\fwd u_t)'}_\H\leq\norm{(\fwd v_{\tau})'}_\H|S'(t)|\leq\frac{R(\tau)}{\tau}|S'(t)|.$$
Now from~$S(t)=tR(t)^{2-\expd}$ we find that 
\begin{align*}
|S'(t)|&=S'(t)=t(2-\expd)R(t)^{-\expd}\langle\fwd u_t-f,(\fwd u_t)'\rangle+R(t)^{2-\expd}\\
&=R(t)^{-\expd}\left[t(2-\expd)\langle\fwd u_t-f,(\fwd u_t)'\rangle+R(t)^2\right].
\end{align*}
Consequently, if we use $R(\tau)=\norm{\fwd v_\tau-f}_\H=\norm{\fwd u_t-f}_\H=R(t)$, the definition of~$S$, and $\tau=S(t)$ we infer
\begin{align*}
\norm{(\fwd u_t)'}_\H&\leq\frac{R(t)}{tR(t)^{2-\expd}}|S'(t)|=\frac{1}{tR(t)}\left[t(2-\expd)\langle\fwd u_t-f,(\fwd u_t)'\rangle+R(t)^2\right]\\
&\leq(2-\expd)\norm{(\fwd u_t)'}_\H+\frac{R(t)}{t}.
\end{align*}
Reordering yields the first inequality in~\eref{ineq:Av_t'_gen}, from where on we proceed as before.
\end{proof}

\subsection{Bounded variation of the forward solution path for $\expd =\expr=1$}

Using the equivalence of the problems together with the Lipschitz regularity of minimizers of the quadratic problem one can at least show that the forward solution path~$t\mapsto\fwd u_t$ for~$\expd=\expr=1$, which is well-defined almost everywhere according to Corollary~\ref{cor:aae-uniqueness}, has \emph{bounded variation}.

\begin{prop}\label{prop:BV}
The solution path~$t\mapsto \fwd u_t$ where~$u_t$ is the minimizer of~$E^{1,1}_t(\cdot;f)$ is of bounded variation on~$(\delta,\infty)$ for all $\delta>0$ and on $[0,\infty)$ if $t_*>0$ holds. Furthermore, the jump part of the measure~$(\fwd u_t)'$ is supported in $[t_*,t_{**}]$.
\end{prop}
\begin{proof}
First, we notice that $\fwd u_t$ is well-defined for almost every $t>0$ according to Remark~\ref{cor:aae-uniqueness}.

Let us first assume that~$t_*>0$, i.e, conditions~\eref{eq:range_cond} and~\eref{eq:source_cond} hold. We already know that~$\fwd u_t$ has zero variation on~$(0,t_*)$ and~$(t_{**},\infty)$. Hence, it is enough to assert finite variation on the interval~$(t_*,t_{**})$. To this end, let~$t_*=t_1<t_2<\dots<t_{n-1}<t_n=t_{**}$ be a finite partition of the interval~$(t_*,t_{**})$. By Theorem~\ref{thm:relation}, we can choose numbers~$0\leq\tau_1<\dots<\tau_n$ such that~$\fwd u_{t_k}=\fwd v_{\tau_k}$ for all~$k=1,\dots,n$. Here,~$\tau_n=\tau_{**}$ is given by the finite extinction time of minimizers of~$E^{2,1}_\tau(\cdot;f)$. Furthermore,  using~\eref{ineq:lipschitz_w_source} with~$\expd=2$ we compute
\begin{align*}
\sum_{k=1}^{n-1}\norm{\fwd u_{t_{k+1}}-\fwd u_{t_k}}_\H&=\sum_{k=1}^{n-1}\norm{\fwd v_{\tau_{k+1}}-\fwd v_{\tau_k}}_\H\leq C_S\sum_{k=1}^{n-1}(\tau_{k+1}-\tau_k)
\leq C_S\tau_{**}<\infty.
\end{align*}
Forming the supremum over all partions of~$(t_*,t_{**})$ shows that~$\fwd u_t$ has bounded variation. If $t_*=0$ one uses \eref{ineq:lipschitz} to deduce the weaker result. Consequently, the finite Radon measure~$(\fwd u_t)'$ can be decomposed into an absolutely continuous part, a jump part, and a Cantor part (see~\cite{ambrosio2000functions} for precise definitions), where the jump part is supported in~$[t_*,t_{**}]$ since~$\fwd u_t$ is constant outside this interval.
\end{proof}

Once more, we obtain statements concerning the subgradient and the solution path.

\begin{cor}
Under the conditions of Theorem~\ref{thm:dv_p<1} or Proposition~\ref{prop:BV}, respectively, it holds:
\begin{enumerate}
\item The map~$t\mapsto p_t$, where~$p_t$ is given by the optimality conditions~\eref{eq:opt_conf_f} and~\eref{eq:opt_cond_u}, has the same regularity as the forward solution path.
\item If~$\fwd$ is bounded from below, meaning that there is $c>0$ such that $c\norm{u}_\X\leq\norm{\fwd u}_\H,\;\forall u\in\X$, then the solution path~$t\mapsto u_t$ has the same regularity as the forward solution path.
\end{enumerate}
\end{cor}

\section{Nonlinear spectral representations}\label{sec:spectral}
\LB{In order to define a nonlinear spectral representation $\phi_t$ of some data~$f$ with respect to the functional~$E_t^{\expd,\expr}(\cdot;f)$, we draw our motivation from classical linear Fourier analysis and follow an axiomatic approach. Formally, the Fourier transform of a sine or cosine -- being eigenfunctions of the negative Laplacian -- is given by a delta distribution which is concentrated on the corresponding eigenvalue (or the frequency after a change of variables). Hence, also in the nonlinear setting eigenfunctions should give rise to atoms in the spectral representation. In addition, in analogy to the inverse Fourier transform, there should be an inverse transform, mapping a nonlinear spectral representation back to the data and allowing for spectral filtering.} 
\subsection{Solution path of generalized singular vectors}\label{sec:singular_vectors}
\LB{To find a nonlinear spectral representation with above noted properties we follow the approach of Gilboa, first brought up in \cite{gilboa2013spectral}}, and examine the solution path that corresponds to singular vectors (cf.~\cite{benning2013ground,benning2018modern}) of~$J$, i.e.,~$f=\fwd\dataX$ where~$\lambda \fwd^*\fwd\dataX\in\partial J(\dataX)$ for some~$\lambda>0$. For such data, one would like to have a delta-peak in the spectral representation to indicate that only one singular vector is contained in the data, that is,~$\phi_t=f\delta_{{1}/{\lambda}}(t)$, where~${1}/{\lambda}$ can be interpreted as a generalized frequency. 

The following proposition characterizes the solution paths of singular vectors with eigenvalue $\lambda>0$.

\begin{prop}\label{prop:singular_vectors}
Let~$\lambda>0$ and~$\dataX\in\H$ such that~$f=\fwd \dataX$ and~$\lambda \fwd^*\fwd\dataX\in\partial J(\dataX)$, i.e.,~$\dataX$ is a singular vector with singular value~$\lambda$. Letting $\mathbf{1}$ denote the indicator function (cf.~\eref{eq:ind_func}), a minimizer~$u_t$ of~$E_t^{\expd,1}(\cdot;f)$ is given by
\begin{align*}
u_t=
\begin{cases}
\mathbf{1}_{(0,(\lambda\norm{f}_\H)^{-1})}(t)\,\dataX,\quad&\expd=1,\\
(1-(t\lambda)^\frac{1}{\expd-1}\norm{f}_\H^\frac{2-\expd}{\expd-1})_+\dataX,\quad&\expd>1,
\end{cases}
\end{align*}
and a minimizer for~$E_t^{2,2}(\cdot;f)$ is given by~$u_t={1}/{(1+t\lambda^2)}\dataX$. The extinction times of these solution are given by~$(\lambda\norm{f}^{2-\expd}_\H)^{-1}$ for~$\expd\geq1$ and~$\infty$ for~$\expd,\expr=2$, respectively.
\end{prop}
\begin{proof}
In the case~$\expd=1$ one can easily check that~$t_*=t_{**}=1/(\lambda\norm{f}_\H)$ if~$f$ is a singular vector. The other minimizers can be obtained by inserting the ansatz $u_t=c(t)\dataX$ into the optimality condition~\eref{eq:opt_cond_u}.
\end{proof}

\Fref{fig:singular_vectors} shows the corresponding solution paths for a singular vector~$\dataX$ with singular value~$\lambda$ such that~$f=\fwd\dataX$ has unit norm and~$\expr=1$. In this case, all paths extinct in~$1/\lambda$. Hence, in order to obtain~$\phi_t=f\delta_{1/\lambda}(t)$, suitable spectral representations for~$\expr=1$ are~$\phi_t=-(\fwd u_t)'$ if~$\expd=1$ and~$\phi_t=t(\fwd u_t)''$ if~$\expd=2$. If~$\fwd$ is bounded from below such that the solution path $t\mapsto u_t$ has the same regularity as the forward solution path, one can even choose~$\phi_t=-u_t'$ or~$\phi_t=tu_t''$, respectively. For other~$\expd$'s an integer derivative does typically not produce a delta peak and one could consider fractional derivatives as done in \cite{cohen2018shape}. Note that by these definitions and due to the finite extinction time the reconstruction formula
\begin{align}\label{eq:reconstruction}
f = \int_0^\infty \phi_t\,dt+\fwd\calP^\fwd(f)
\end{align}
holds which can be used for {spectral filtering} by defining
\begin{align}\label{eq:filtering}
f_F\defi\int_0^\infty F(t)\phi_t\,dt+F(\infty)\fwd\calP^\fwd(f),
\end{align}
where~$F$ is a sufficiently well-behaved filter function (cf.~\cite{gilboa2014total}, for instance).
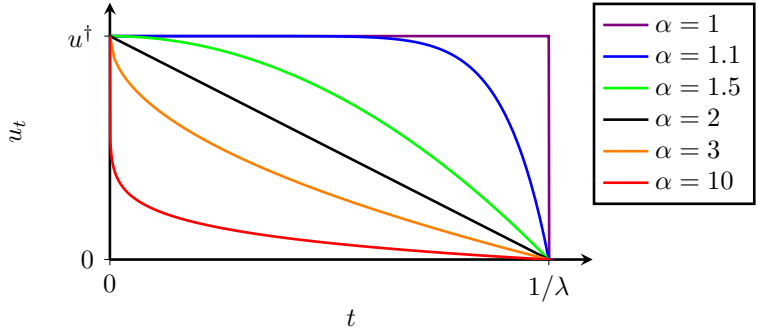
\begin{figure}[h!]
\centering
\begin{tikzpicture}[scale=1]
\begin{axis}[
	x tick style={color=black},
	y tick style={color=black},
	xmin = 0,
	xmax = 1.1,
	ymin = 0,
	ymax = 1.15,
	width = 8cm,
	height = 5cm,
    axis lines = left,
    xlabel = $t$,
    ylabel = {$u_t$},
    xtick={0,1},
    xticklabels={0,$1/\lambda$},
    ytick={0,1},
    yticklabels={0,$\dataX$},
    legend cell align={left},
    legend style={at={(1,1)},anchor=north west},
    line width = 1pt
]
\addplot [
    domain=0:1.01, 
    samples=1000, 
    color=violet,
]
{(x<=1)};
\addlegendentry{$\expd=1\phantom{.0}$}
\addplot [
    domain=0:1, 
    samples=100, 
    color=blue,
    ]
    {(1-(x)^10)};
\addlegendentry{$\expd=1.1$} 
\addplot [
    domain=0:1, 
    samples=100, 
    color=green,
]
{(1-x^2)};
\addlegendentry{$\expd=1.5$}
\addplot [
    domain=0:1, 
    samples=100, 
    color=black,
    ]
    {(1-x)};
\addlegendentry{$\expd=2\phantom{.0}$}
\addplot [
    domain=0:1, 
    samples=100, 
    color=orange,
    ]
    {(1-(x)^0.5)};
\addlegendentry{$\expd=3\phantom{.0}$} 
\addplot [
    domain=0:1, 
    samples=1000, 
    color=red,
    ]
    {(1-(x)^(1/9))};
\addlegendentry{$\expd=10$} 
\end{axis}
\end{tikzpicture}
\caption{Solution paths of normalized singular vectors for different values of $\expd$ and $\expr=1$\label{fig:singular_vectors}}
\end{figure}

\begin{rem}\label{rem:distribution}
Note that while~$\phi_t=-(\fwd u_t)'$ is a well-defined finite Radon measure according to Proposition~\ref{prop:BV}, whereas this is a-priori unclear for~$\phi_t=t(\fwd u_t)''$. However, due to the finite extinction time, this spectral representation can be defined in a distributional sense, via
\begin{align}
\phi_t(\psi)\defi-\int_0^\infty \langle(\fwd u_t)',(t\psi(t))'\rangle~dt,
\end{align}
where~$\psi:\R\to\H$ is a Fr\'{e}chet-differentiable test function with~$\psi(t)=0$ for all~$t$ in a neighborhood of~$0$. Owing to Theorem~\ref{thm:dv}, the second condition is not even necessary if one of the conditions~\eref{eq:range_cond} or~\eref{eq:source_cond} holds since in that case~$\norm{(\fwd u_t)'}_\H$ is integrable in zero.
\end{rem}

Proposition~\ref{prop:singular_vectors} also shows that, although all problems for~$\expd>1$ are equivalent, they significantly differ in terms of the spectral representations which can be obtained from their solution paths. Furthermore, since the minimizer for~$\expr=2$ smoothly depends on~$t$, no singular spectral representation can be achieved by computing time derivatives which is why we will restrict ourselves to the case~$\expr=1$ for the rest of the manuscript.

Another interesting consequence of Proposition~\ref{prop:singular_vectors} is that some of the models~$E_t^{\expd,1}(\cdot;f)$ are \emph{scale invariant} on eigenfunctions. To see this, we choose~$J=\tv$ as the total variation of functions on~$\R^n$,~$\X=\bv(\R^n)\cap L^2(\R^n)$,~$\H=L^2(\R^n)$, and~$\fwd$ the continuous embedding operator. It is well-known that eigenfunctions of~$\tv$ are given by indiactor functions of so called calibrable sets~$\Omega\subset\R^n$ with eigenvalue~$P(\Omega)/|\Omega|$ where~$P$ denotes the perimeter and~$|\cdot|$ is the~$n$-dimensional Lebesgue measure (cf.~\cite{bellettini2002total,alter2005characterization}). If~$f=\mathbf{1}_\Omega$ for calibrable~$\Omega$, we find that the extinction time of minimizers~$E_t^{\expd,1}(\cdot;f)$ is given by~$t_\mathrm{ext}(\Omega)={|\Omega|}^\frac{\expd}{2}/P(\Omega)$ for~$\expd\geq 1$. If one rescales~$\Omega_r=r\Omega$ with some~$r>0$, then~$\Omega_r$ is still calibrable and the extinction time changes to $t_\mathrm{ext}(\Omega_r)=r^\frac{n(\expd-2)+2}{2}t_ \mathrm{ext}(\Omega).$ Hence, we observe that for any dimension~$n\geq2$ there is~$\expd\defi2-2/n\in[1,2)$ such that~$t_ \mathrm{ext}(\Omega_r)=t_ \mathrm{ext}(\Omega)$ which makes the model \emph{scale invariant}. Note that in dimension~$n=2$, which is most relevant for imaging applications, the model~$E^{1,1}(\cdot;f)$ becomes both contrast and scale invariant.

\subsection{Spectral representations for~$\expd=1$ and~$\expd=2$}
From now on our setting will be a Gelfand-triple~$\X\hookrightarrow\H\hookrightarrow\dualX$ such that operator~$\fwd$ becomes a continuous embedding operator and will thus be omitted in our notation. In the absence of a forward operator one usually refers to singular vectors as \emph{eigenvectors}. Due to the observations in the previous section, we will only study the functionals~$F_\tau(\cdot;f)\defi E^{2,1}_\tau(\cdot;f)$ and~$E_t(\cdot;f)\defi E_t^{1,1}(\cdot;f)$ and fix our notation in such a way that the corresponding minimizers are denoted by~$v_\tau$ and~$u_t$, respectively. We consider the spectral representations given by~$\varphi_\tau\defi\tau v_\tau''$, which is to be understood in the distributional sense, and~$\phi_t\defi-u_t'$, the latter being a finite Radon measure according to Proposition~\ref{prop:BV}.

Next we formulate a theorem which is a generalization of Proposition~\ref{prop:singular_vectors} and deals with an important question concerning nonlinear spectral decompositions, namely with the decomposition of a linear combination of generalized eigenvectors. Two conditions that suffice for a perfect decomposition into eigenvectors are the (SUB0) condition and orthogonality of the eigenvectors, introduced in~\cite{schmidt2018inverse}. Here the authors showed that the inverse scale space flow is able to decompose the data perfectly into the eigenvectors. A similar statement holds true for the variational problem~$F_\tau(\cdot;f)$, in particular, the solution path $v_\tau$ will shrink each eigenvector linearly until disappearance and will, thus, be piecewise affine in $\tau$.

\minrev{For a more compact notation we will from now on abbreviate $K:=\partial J(0)$, which can be viewed as characteristic set of $J$ since it contains all subdifferentials and defines $J$ via duality (cf.~\eref{eq:subdiff} and \eref{eq:J_dual}).}

\begin{thm}[Linear combination of eigenvectors I]\label{thm:linear_combination}
Let~$f$ be the the linear combination of orthogonal eigenvectors, i.e.,~$f=\sum_{i=1}^n\gamma_i u_i$ where~$\gamma_i\neq 0$,~$\lambda_iu_i\in\partial J(u_i)$ with~$\lambda_i>0$, and~$\langle u_i,u_j\rangle=0$ for all~$i,j=1,\dots,n$,~$i\neq j$. Furthermore, we define~$p_k\defi\sum_{i=k}^n\sgn(\gamma_i)\lambda_iu_i$ and assume that 
\begin{align}\label{cond:sub0}\tag{SUB0}
p_k\in K,\quad k=1,\dots,n.
\end{align}
Additionally, we assume an ordering such that~$|\gamma_i|/\lambda_i<|\gamma_{i+1}|/\lambda_{i+1}$ holds for all~$i=1,\dots,n$. Then the minimizer~$v_\tau$ of~$F_\tau(\cdot;f)$ is given by 
\begin{align}\label{eq:quadr_min}
v_\tau=\sum_{i=k}^n\sgn(\gamma_i)(|\gamma_i|-\tau\lambda_i)u_i,\quad\tau_{k-1}<\tau\leq\tau_k,
\end{align}
where~$\tau_0\defi 0$,~$\tau_k\defi\gamma_k/\lambda_k$, and~$k=1,\dots,n$.
\end{thm}
\begin{proof}
The proof works along the same lines as the proof of \cite[Thm.~3.14]{schmidt2018inverse}.
\end{proof}

\begin{rem}
Note that it is straightforward to extend this result to data which is composed of generalized singular vectors, i.e.,~$f=\sum_{i=1}^n\gamma_i\fwd u_i$. To this end, one has to demand~$\fwd$-orthogonality~$\langle\fwd u_i,\fwd u_j\rangle=0$ for~$i\neq j$ and define~$p_k\defi \sum_{i=k}^n\sgn(\gamma_i)\lambda_i\fwd^*\fwd u_i$, instead. 
\end{rem}

\begin{rem}
It is no significant restriction in Theorem~\ref{thm:linear_combination} to assume that all~$|\gamma_i|/\lambda_i$ are different for~$i=1,\dots,n$. If this were not the case, the corresponding eigenvectors would simply shrink away simultaneously. However, in order to avoid unnecessarily complicated formulae, we refrained from considering this case.
\end{rem}

\begin{rem}[Action of proximal operators]
Theorem~\ref{thm:linear_combination} can be interpreted in such a way that if the data~$f$ can be written as a linear combination of orthogonal eigenvectors fulfilling~\eref{cond:sub0}, then the proximal operator~$\prox_{\tau J}(f)\defi\argmin_v\left\lbrace\frac{1}{2}\norm{v-f}^2+\tau J(v)\right\rbrace$ performs shrinkage on the eigendirections. This is in particular true for the~$\norm{\cdot}_1$-norm where the standard basis of~$\R^n$ constitutes a set of orthogonal eigenvectors fulfilling~\eref{cond:sub0}.
\end{rem}

\begin{example}
Let us illustrate the preceding remark for the proximal operator of the~$\infty$-norm in two dimensions. Let 
$$v_\tau\defi\mathrm{prox}_{\tau\|\cdot\|_\infty}(f)\defi\argmin\left\lbrace\frac{1}{2}\norm{v-f}^2+\tau\norm{v}_\infty\st v\in\R^2\right\rbrace$$
and~$K\defi\{v\in\R^2\st\norm{v}_1\leq1\}$ be the unit ball of the~$1$-norm. We observe that  $u_1=(1,1)^T/2$ and $u_2=(-1,1)^T/2$ constitute a basis of eigenvectors of~$\norm{\cdot}_\infty$ with eigenvalue~$1$. In particular, any~$f\in\R^2$ can be written as
$$f=(f_1,f_2)^T=(f_1+f_2)u_1+(f_2-f_1)u_2=:\gamma_1u_1+\gamma_2u_2.$$
Note that the~\eref{cond:sub0} condition is met since~$u_1,u_2,u_1+u_2\in K$ and the~$u_i$'s are orthogonal. If~$f$ is an eigenvector of~$\norm{\cdot}_\infty$, the analytic expression for~$\prox_{\tau\norm{\cdot}_\infty}(f)$ becomes trivial and, thus, we assume that~$\gamma_1,\gamma_2\neq 0$ and~$|\gamma_1|\neq|\gamma_2|$. This guarantees that~$f$ is no eigenvector. Furthermore, we reorder such that~$0<|\gamma_1|<|\gamma_2|$ holds. Hence, we find by~\eref{eq:quadr_min} that~$v_\tau$ is given by
$$
v_\tau=\begin{cases}
\mathrm{sgn}(\gamma_1)(|\gamma_1|-\tau)u_1+\mathrm{sgn}(\gamma_2)(|\gamma_2|-\tau)u_2,\quad &0\leq\tau\leq\tau_1\defi|\gamma_1|,\\
\mathrm{sgn}(\gamma_2)(|\gamma_2|-\tau)u_2,\quad&\tau_1<\tau\leq\tau_2\defi|\gamma_2|.
\end{cases}
$$
\end{example}

\begin{cor}[Linear combination of eigenvectors II]\label{cor:linear_combination}
Under the conditions of Theorem~\ref{thm:linear_combination} the minimizer~$u_t$ of~$E_t(\cdot;f)$ is~$u_t=v_{S(t)}$ where~$S$ is given by
\begin{align}\label{eq:S_linear_comb}
S(t)=\frac{t\sqrt{\sum_{i=1}^{k-1}\gamma_i^2\norm{u_i}_\H^2}}{\sqrt{1-t^2\norm{p_k}_\H^2}},\quad t_{k-1}<t\leq t_k,\quad k=1,\dots,n.
\end{align}
Here,~$t_k\defi T(\tau_k)=\tau_k/R(\tau_k)$ for~$k=0,\dots,n$ and the~$S(t)\defi0$ if~$k=1$.
\end{cor}
\begin{proof}
From the definition of~$v_\tau$ in~\eref{eq:quadr_min} we easily see, using the orthogonality of the~$u_i$'s, that 
$$T(\tau)=\frac{\tau}{\sqrt{\sum_{i=1}^{k-1}\gamma_i^2\norm{u_i}_\H^2+\tau^2\norm{p_k}_\H^2}},\quad\tau_{k-1}<\tau\leq\tau_k,\quad k=1,\dots,n.$$ 
Inverting this on the intervals~$(\tau_{k-1},\tau_k)$ for~$k\geq 2$ yields the expression for~$S$. Furthermore, it holds that~$t_k=T(\tau_k)<{1}/{\norm{p_k}_\H}$ for~$k\geq 2$ which makes~$S$ well-defined and continuous. Noting that $S$ is the inverse of $T(\tau)$ for $\tau>\tau_1$ and applying Lemmas~\ref{lem:u_t-to-v_tau} and \ref{lem:v_tau-to-u_t} shows that $u_t=v_{S(t)}$.  
\end{proof}
Now we investigate the spectral representations~$\phi_t$ and~$\varphi_\tau$ under the conditions of Theorem~\ref{thm:linear_combination}. By means of Corollary~\ref{cor:linear_combination}, we find
$$\phi_t=-u_t'=-\frac{d}{dt}\sum_{i=k}^n\sgn(\gamma_i)(|\gamma_i|-S(t)\lambda_i)u_i$$
for~$t_{k-1}<t<t_k$. From~\eref{eq:S_linear_comb} it is obvious that~$S$ is continuously differentiable on the intervals~$(t_{k-1},t_k)$ and discontinuous only in~$t_1$. Hence, the measure~$\phi_t$ is singular only in~$t_*\defi t_1=T(\tau_1)=1/\norm{p_1}$ and, since~$S$ is continuously differentiable on~$(t_{k-1},t_k)$, represented by a bounded function, elsewhere. The jump of~$u_t$ in~$t_*$ is given by~$f-v_{\hat{\tau}}$, where~$\hat{\tau}\defi\tau_1$, and hence the singular part of $\phi_t$ reduces to
\begin{align}\label{eq:spectrum_1h}
\phi_{t_*}=f-v_{\hat{\tau}}=\gamma_{1}u_{1}+\sum_{i=2}^n\hat{\tau}\sgn(\gamma_i)\lambda_iu_i.
\end{align}
This can be considered bad news since, on one hand, the spectral representation~$\phi$ of the contrast-invariant problem~$E_t(\cdot;f)$ is not able to isolate an individual eigenvector although it has a delta peak at~$t_*$. On the other hand, the time point~$t_*$ where the peak occurs is independent of the specific eigenvector that vanishes. Thus, it cannot be brought into correspondence with the eigenvalue~$\lambda_1$ or the factor~$\gamma_1$. In contrast, the spectral representation~$\varphi$ is given by
\begin{align}\label{eq:spectrum_qu}
\varphi_\tau=\sum_{k=1}^n\gamma_k u_k\delta_{\tau_k}(\tau),\quad\tau>0
\end{align}
which is a sum of singular Dirac measures and hence a perfect decomposition of the data~$f$ into its components.

\subsection{Affine solution paths of the quadratic problem}\label{sec:affine}

\LB{Theorem \ref{thm:linear_combination} in particular states that if the data $f$ is a linear combination of eigenvectors satisfying additional fairly strong conditions, the corresponding solution path $v_\tau$ is piecewise affine in the time variable. In \cite{burger2016spectral} this has been proven in finite dimension under the condition that $J$ is a \emph{polyhedral semi-norm}. In infinite dimensions and for general data $f$ this behavior cannot be expected. However, we would like to find a condition which assures that the solution path~$v_\tau$ is affine in~$\tau$ at least on a small interval $[0,\hat{\tau}]$. Due to Theorem~\ref{thm:non-unique-affine} this is in one-to-one correspondence to an exact penalization effect of the corresponding contrast invariant problem $E_t^{1,1}(\cdot,f)$ and, hence, to the validity of conditions \eref{eq:range_cond} and \eref{eq:source_cond}. We start with equivalent reformulations of this behavior and give several illustrative examples in finite and infinite dimensions.}

By Moreau's identity (cf.~\cite{rockafellar2015convex} for a finite dimensional version), we find that the minimizer~$v_\tau$ of~$F_\tau(\cdot;f)$ is given by
\begin{align}\label{eq:moreau}
v_\tau=f-\tau\proj_K\left({f}/{\tau}\right).
\end{align}
Here we used that~$J=\chi_K^*$ (cf.~\eref{eq:J_dual},~\eref{eq:char_func}) and let~$\proj_K(\cdot)$ denote the projection on the closed and convex set~$K$ with respect to the Hilbert norm~$\norm{\cdot}_\H$ which is well-defined as $K\cap\H\ni0$.  
\begin{rem}
While Moreau's identity is often formulated in Hilbert spaces or finite dimensions, the identity~$p\in\partial J(u)\iff u\in\partial J^*(p)$, which holds for lower semi-continuous and convex~$J$ defined on a Banach space~$\X$ (cf.~\cite[Ch.~5]{lucchetti2006convexity}), makes it easy to show that it is applicable also in our slightly more general setting. 
\end{rem}
The beauty of the representation~\eref{eq:moreau} lies in the fact that it allows us to study the solution path~$v_\tau$ by investigating the geometric properties of the set~$K$ and the projection onto it.

Using~\eref{eq:moreau}, the residual is given by~$R(\tau)=\tau\norm{\proj_K\left({f}/{\tau}\right)}_\H$ and therefore~$T(\tau)={\norm{\proj_K\left({f}/{\tau}\right)}_\H}^{-1}$. Taking Theorem~\ref{thm:non-unique-affine} into account, the following statements are equivalent:
\begin{subequations}\label{eq:jump_crit}
\begin{equation}
\proj_K(f/\tau)\in\argmin\{\norm{p}_\H\st p\in\partial J(f)\},\quad\forall0<\tau\leq\hat{\tau},
\end{equation}
\begin{equation}
\tau\mapsto v_\tau\defi f-\tau\proj_K(f/\tau)\text{ is affine for }0<\tau\leq\hat{\tau},
\end{equation}
\begin{equation}
t\mapsto T(\tau)\text{ is constant on }(0,\hat{\tau}].
\end{equation}
\end{subequations}
Note that~\eref{eq:jump_crit} is always fulfilled if~$K\subset\R^n$ is polyhedral\footnote{Polyhedral in this context means being the convex hull of a finite set of vectors.} since in this case the solution~$v_\tau$ is piecewise affine with~$v_\tau=f-\tau p$ for~$\tau\in[0,\hat{\tau}]$ and~$p\in\partial J(f)$, as it was shown in~\cite{burger2016spectral} or less general for LASSO /~$\ell^1$ problems in~\cite{donoho1fast,rosset2007piecewise,tibshirani2011solution,bringmann2018homotopy}. However, the condition of a polyhedral~$K$ is neither necessary nor can it be completely waived, as the following examples show.  

\begin{example}\label{ex:ellipse}
Let~$a_1,a_2>0$ with~$a_1\neq a_2$, let~$M=\mathrm{diag}(a_1,a_2)$, and~$J(u)=\sqrt{\langle u,Mu\rangle}$. Then,~$K$ is an ellipse with semi-axes~$\sqrt{a_1}$ and~$\sqrt{a_2}$ and, therefore, not polyhedral. Here,~$\partial J(f)=\lbrace(a_1f_1,a_2f_2)/J(f)\rbrace$ for~$f\neq(0,0)$. If~$f$ is no eigenvector, i.e.,~$f$ is not parallel to a semi-axis, the projection of~$f/\tau$ onto~$K$ does not equal~$\partial J(f)$ for any~$\tau>0$, as it can be easily seen from the corresponding Karush-Kuhn-Tucker conditions. Hence, conditions~\eref{eq:jump_crit} are violated and there is no affine behavior. 
\end{example}

\LB{\begin{example}\label{ex:half_ball}
Let now $J:\R^2\to\R$ be given by
$$
J(u)=\begin{cases}
\norm{u}_1,\quad\text{if }\mathrm{sgn}(u_1)=\mathrm{sgn}(u_2)\\
\norm{u}_2,\quad\text{else}.
\end{cases}
$$
\begin{minipage}{0.65\textwidth}
Then $K$ coincides with the unit square in the first and the third quadrant, and with the unit circle in the remaining quadrants of $\R^2$ (\fref{fig:non-polyhedral_K}). It is easy to see that all vectors in the second and fourth quadrant are eigenvectors and hence \eref{eq:jump_crit} trivially holds. The solution path of vectors in the first and third quadrant is also piecewise affine since the problem coincides with standard $\ell^1$-shrinkage (see the references before) there. Note that $K$ is not polyhedral either.
\end{minipage}
\hfill
\begin{minipage}{0.35\textwidth}
\centering
\begin{tikzpicture}
	\fill[pattern=horizontal lines light gray] (0,0) rectangle (1,1);
   	\fill[pattern=horizontal lines light gray] (0,0) rectangle (-1,-1);
   	\fill[pattern=horizontal lines light gray] (0,0) circle (1);
   	\draw[->, thick] (-1.5,0) -- (1.5,0) node[right] {$u_1$};
    \draw[->, thick] (0,-1.5) -- (0,1.5) node[above] {$u_2$};
	\draw [black,ultra thick](1,0)--(1,1)--(0,1);
   	\draw [black,ultra thick,domain=90:180] plot ({cos(\x)}, {sin(\x)});
	\draw [black,ultra thick](-1,0)--(-1,-1)--(0,-1);
	\draw [black,ultra thick,domain=270:360] plot ({cos(\x)}, {sin(\x)});
	\draw (1,1) node [anchor=south west] {$K$};
	\node at (1,0) [circle,fill,inner sep=1.5pt]{};
	\node at (-1,0) [circle,fill,inner sep=1.5pt]{};
	\node at (0,1) [circle,fill,inner sep=1.5pt]{};
	\node at (0,-1) [circle,fill,inner sep=1.5pt]{};
	\draw (1,0) node [anchor=north west] {$1$};
	\draw (0,1) node [anchor=south west] {$1$};
	\draw (0,-1) node [anchor=north west] {$-1$};
	\draw (-1,0) node [anchor=north east] {$-1$};
\end{tikzpicture}

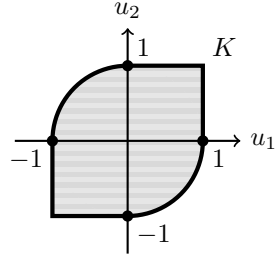
\captionof{figure}{Non-polyhedral $K$\label{fig:non-polyhedral_K}}
\end{minipage}
\end{example}}

\begin{example}\label{eq:l1norm}
Let~$\Omega\subset\R^n$ be open and bounded,~$\X=\H=L^2(\Omega)$ and~$J=\norm{\cdot}_1$. Then~$K=\{u\in L^\infty(\Omega)\st\norm{u}_\infty\leq 1\}$ and if~$f\in L^2(\Omega)$ fulfills~$f(x)\geq c>0$ for almost every~$x\in\Omega$, then for~$0<\tau\leq\hat{\tau}=c$ it holds $\proj_K(f/\tau)(x)=1$ for almost every $x\in\Omega$, hence the jump exists. Obviously, here~$K$ is also not polyhedral since the unit ball in~$L^\infty(\Omega)$ is not generated by the convex combinations of a finite number of functions. 
\end{example}

\begin{example}\label{ex:tv}
Let~$I\subset\R$ be an interval,~$\X=\bv(I)$,~$\H=L^2(I)$, and~$J=\tv$. If~$f$ is piecewise constant, then according to~\cite{cristoferi2016exact} the solution~$v_\tau$ is piecewise affine with~$v_\tau=f-\tau {p}$ for~$\tau\in[0,\hat{\tau}]$ and~${p}\in\partial J(f)$. In~\cite{lasica2017total} the authors prove similar results in two dimensions, using anisotropic total variation as regularization and assuming the data to be piecewise constant on rectangles.
\end{example}

The following theorem characterizes and affine solution path for small times.

\begin{thm}[Affine solution path]
Let~$\hat{p}\in\partial J(f)$ with~$\norm{\hat{p}}_\H=s_*$. Then~$\hat{\tau}$ given by
\begin{align}\label{def:hat_tau}
\hat{\tau}\defi 2\inf_{\substack{q\in K\\\norm{q}_\H\leq\norm{\hat{p}}_\H}}\frac{J(f)-\langle q,f\rangle}{\norm{\hat{p}}_\H^2-\norm{q}_\H^2}
\end{align}
is positive, \LB{if and only if} $v_\tau\defi f-\tau \hat{p}$ is the minimizer of~$F_\tau(\cdot;f)$ for~$\tau\in[0,\hat{\tau}]$.
\end{thm}
\begin{proof}
\LB{$v_\tau=f-\tau\hat{p}$ being a minimizer is equivalent to $\proj_K(f/\tau)=\hat{p}$ for $0<\tau\leq\hat{\tau}$. This can be rephrased as $\norm{\hat{p}-f/\tau}_\H^2\leq\norm{q-f/\tau}_\H^2$ for all $q\in K$ and $0<\tau\leq\hat{\tau}$ which is equivalent to 
$$\tau(\norm{\hat{p}}_\H^2-\norm{q}^2_\H)\leq 2(J(f)-\langle q,f\rangle),\quad\forall q\in K,\,0<\tau\leq\hat{\tau}.$$
From here the equivalence with \eref{def:hat_tau} is obvious.}
\end{proof}

The following proposition provides at least a necessary condition for~$\hat{\tau}$ being positive. 

\begin{prop}\label{prop:hat_tau}
Let~$\hat{p}\in\partial J(f)$ with~$\norm{\hat{p}}_\H=s_*$. If~$\hat{p}$ is not an eigenvector and~$\{p\st\langle p,f\rangle = J(f)\}$ is the only supporting hyperplane of~$K$ through~$\hat{p}$, then~$\hat{\tau}=0$.
\end{prop}
\begin{proof}
If~$\hat{p}$ is not an eigenvector, then we know that there is a positive angle between~$\hat p$ and~$f$, i.e., there exists a direction~$\varphi$ orthogonal to~$f$ and~$\delta > 0$ with~$\langle \hat{p}, \varphi \rangle \leq - \delta \Vert \varphi \Vert_\H$. Since the supporting hyperplane is unique, there exists a sequence of directions~$\varphi_n$ with~$\Vert \varphi_n \Vert_\H \rightarrow 0$ --~becoming orthogonal to~$f$ in the limit~--~such that~$q_n=\hat p + \varphi_n \in K$ and
$$  \langle \hat p , \varphi_n \rangle \leq -\frac{\delta}{2} \Vert \varphi_n \Vert_\H,\quad   \frac{\vert \langle \varphi_n, f \rangle\vert }{\Vert \varphi_n \Vert_\H} \rightarrow 0.~$$
Thus, since~$\norm{\varphi_n}_\H<\delta$ for~$n$ large enough,
\begin{align*}
\limsup_{n\to\infty} \frac{J(f) - \langle q_n,f \rangle}{\Vert \hat p \Vert_\H^2 - \Vert q_n \Vert_\H^2} &= \limsup_{n\to\infty} \frac{ \langle \varphi_n, f \rangle }{\Vert \varphi_n \Vert_\H} \frac{\Vert \varphi_n \Vert_\H}{ - 2 \langle \hat p , \varphi_n \rangle -\Vert \varphi_n \Vert_\H^2} \\
&\leq \lim_{n\to\infty} \frac{\vert \langle \varphi_n, f \rangle \vert}{\Vert \varphi_n \Vert_\H} \frac{1}{\delta - \Vert \varphi_n \Vert_\H } = 0.
\end{align*}
\end{proof}

Hence, for sets~$K$ with smooth boundary one will in general not observe a (piecewise) affine behavior of the solution path. This can also be derived from~\cite{holmes1973smoothness} which states that in case of~$\X$ being a Hilbert space the degree of differentiability of the projection map~$\proj_K(\cdot)$ is given by~$d-1$ if~$K$ has a~$C^{d}$-boundary.

\section{Numerical results}\label{sec:numerics}

In the following, we will present numerical experiments that serve to illustrate the theoretical results of this work. The first experiment will use artificially generated data whereas the second one is computed on a real photograph. To be able to compute a spectral representation, we will restrict ourselves to the functionals~$E_t(\cdot;f)$ and~$F_\tau(\cdot;f)$ whose minimization we achieve using the Primal-Dual-Algorithm of Chambolle and Pock~\cite{chambolle2011first}. \LB{For computing the spectral representations, we choose a equidistant sequence of time points and compute the corresponding minimizers with a warm-start initialization. The spectral representations $\phi_t$ or $\varphi_\tau$ are then computed through a first or second order difference quotient, respectively. The complexity of this procedure equal the complexity of solving a parabolic PDE -- like for instance the total variation flow -- via an implicit Euler method.}

\subsection{Sparse deconvolution}

Here, we consider 1D sparse deconvolution of a signal~$f\in\R^n$ which is obtained by convolving a peak signal~$\dataX\in\R^n$ with a gaussian kernel of finite length (cf.~left in \fref{fig:data_deconv}). In this setting,~$\X=\H=\R^n$,~$\fwd$ corresponds to a convolution operator, and~$J$ is given by the 1-norm. The data~$\dataX=-0.1u_1+0.2u_2+0.2u_3-0.4u_4+0.5u_5$ is a linear combination of~$\fwd$-orthogonal singular vectors~$u_i$ all of which have the same singular value~$\lambda\approx 5.137$ and satisfy~\eref{cond:sub0}. The~$u_i$'s simply consist of a single peak of height 1. Note that in this case $\fwd$-orthogonality simply means that the supports of the convolved peaks do not intersect. Hence, we know from Theorem~\ref{thm:linear_combination} and the subsequent remarks that the solution path~$v_\tau$ successively shrinks the singular vectors until their contributions disappear. In particular, there are four critical time points $\tau_i$, $i=1,\dots,4$ -- corresponding to the four different peak heights -- where all peaks of this very height vanish. This is illustrated on the right hand side of \fref{fig:data_deconv}, where the red, pink, and blue markers indicate the height of the corresponding peak at times $\tau_1$, $\tau_2$, and $\tau_3$, respectively. The fourth critical time $\tau_4$ coincides with the extinction time, meaning that the solution is identical to zero. 

\begin{figure}[h!]
\centering
\includegraphics[trim={2cm 2cm 2cm 1.5cm},clip, height=4.8cm, width=0.49\textwidth]{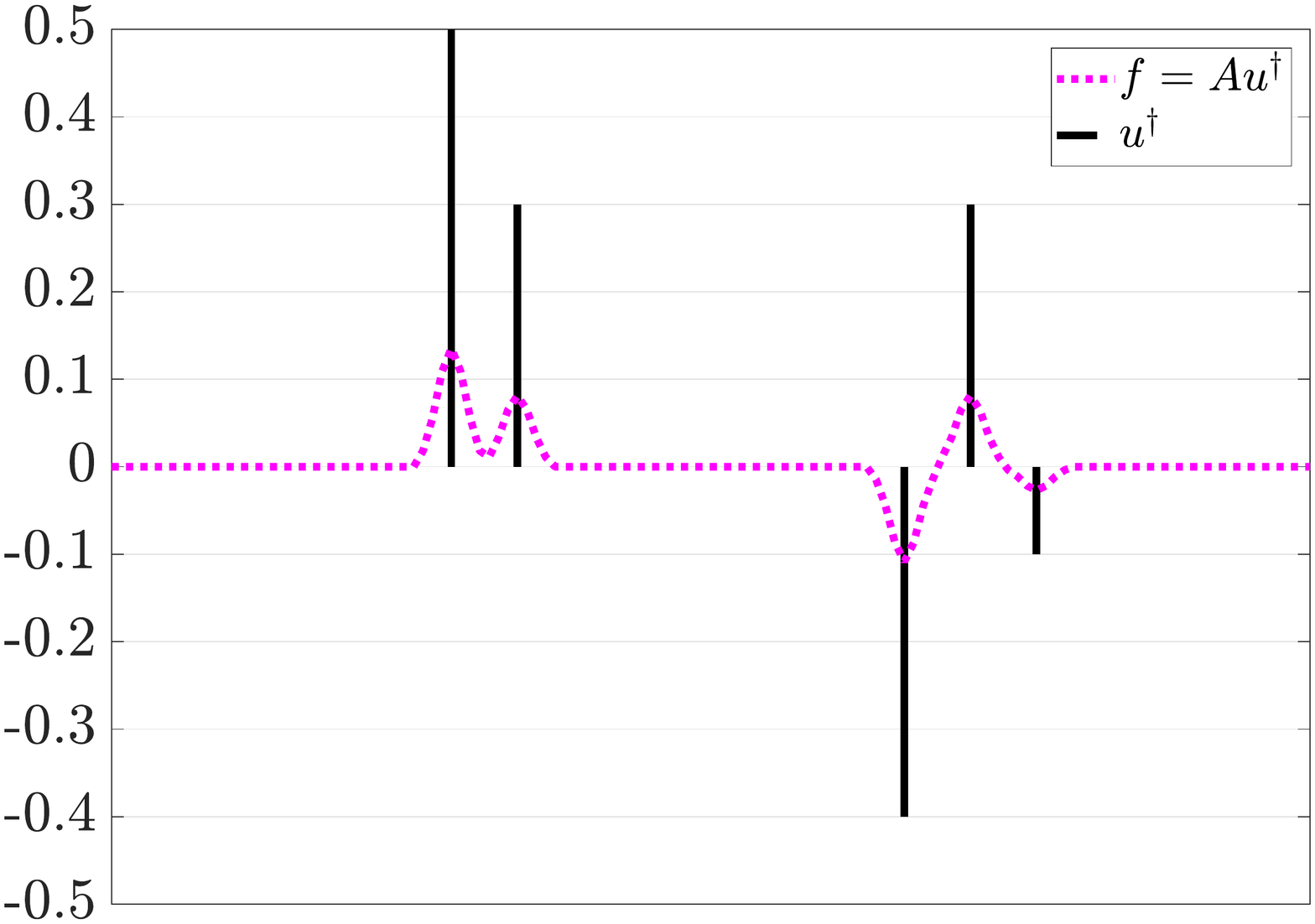}
\hfill
\includegraphics[trim={2cm 2cm 2cm 1.5cm},clip, height=4.8cm, width=0.49\textwidth]{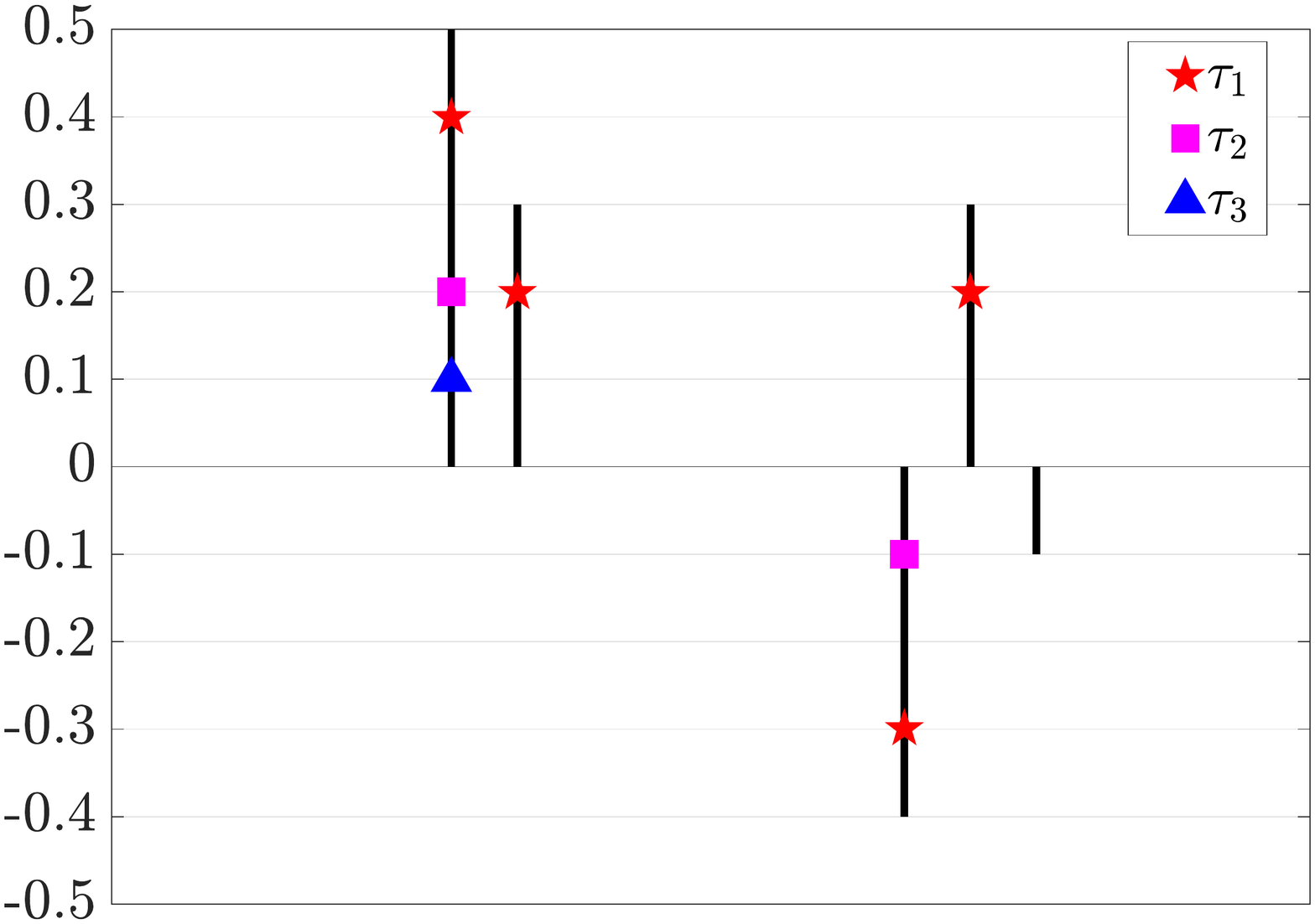}
\caption{\textbf{Left:} Data $\dataX$ and forward data $f$, \textbf{right:} solution at cricital times\label{fig:data_deconv}}
\end{figure}


\begin{figure}[h!]
\centering
\includegraphics[trim={2cm 1cm 1.2cm 0.5cm},clip, height=4cm, width=0.49\textwidth]{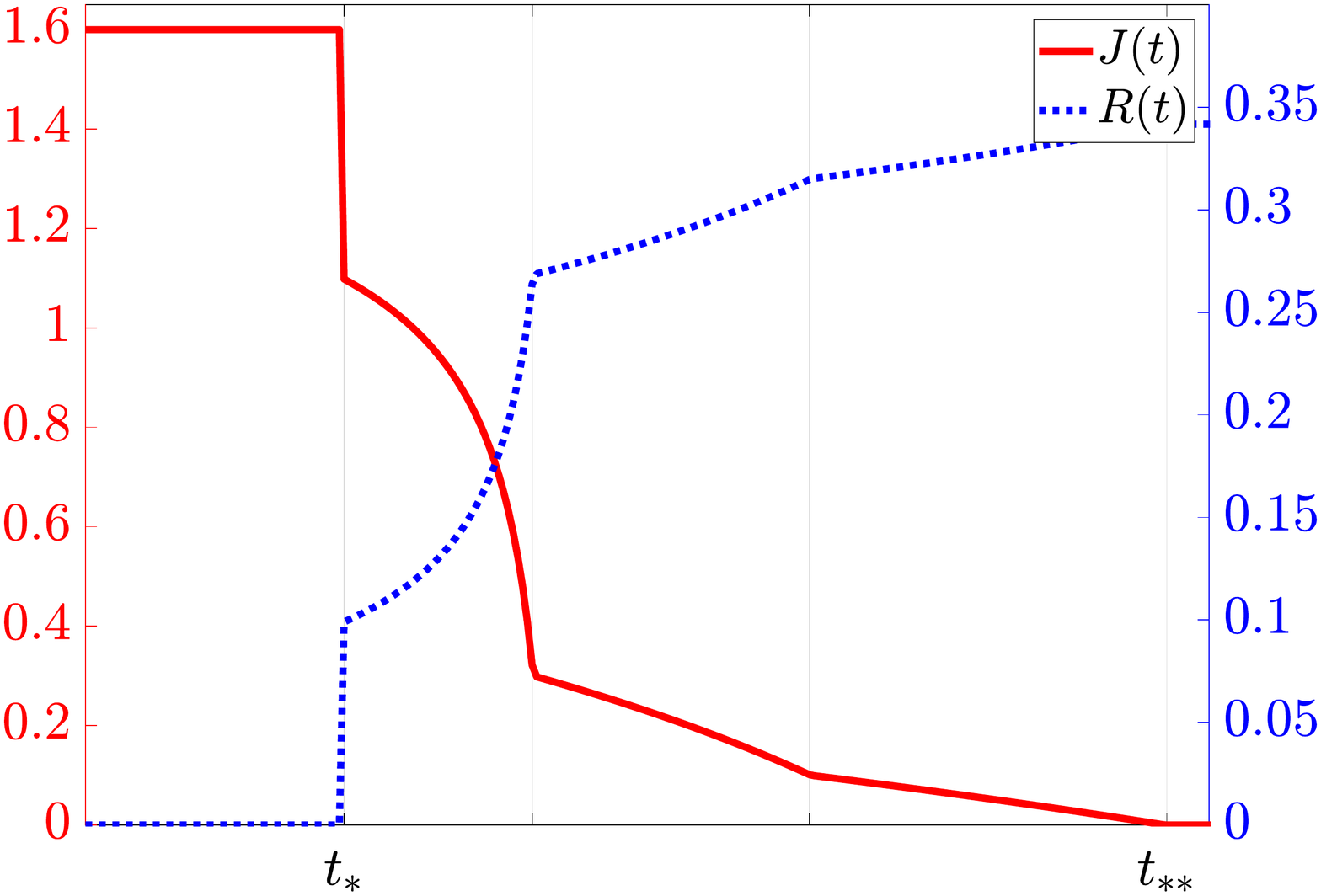}
\hfill
\includegraphics[trim={2cm 1cm 1.2cm 0.5cm},clip, height=4cm, width=0.49\textwidth]{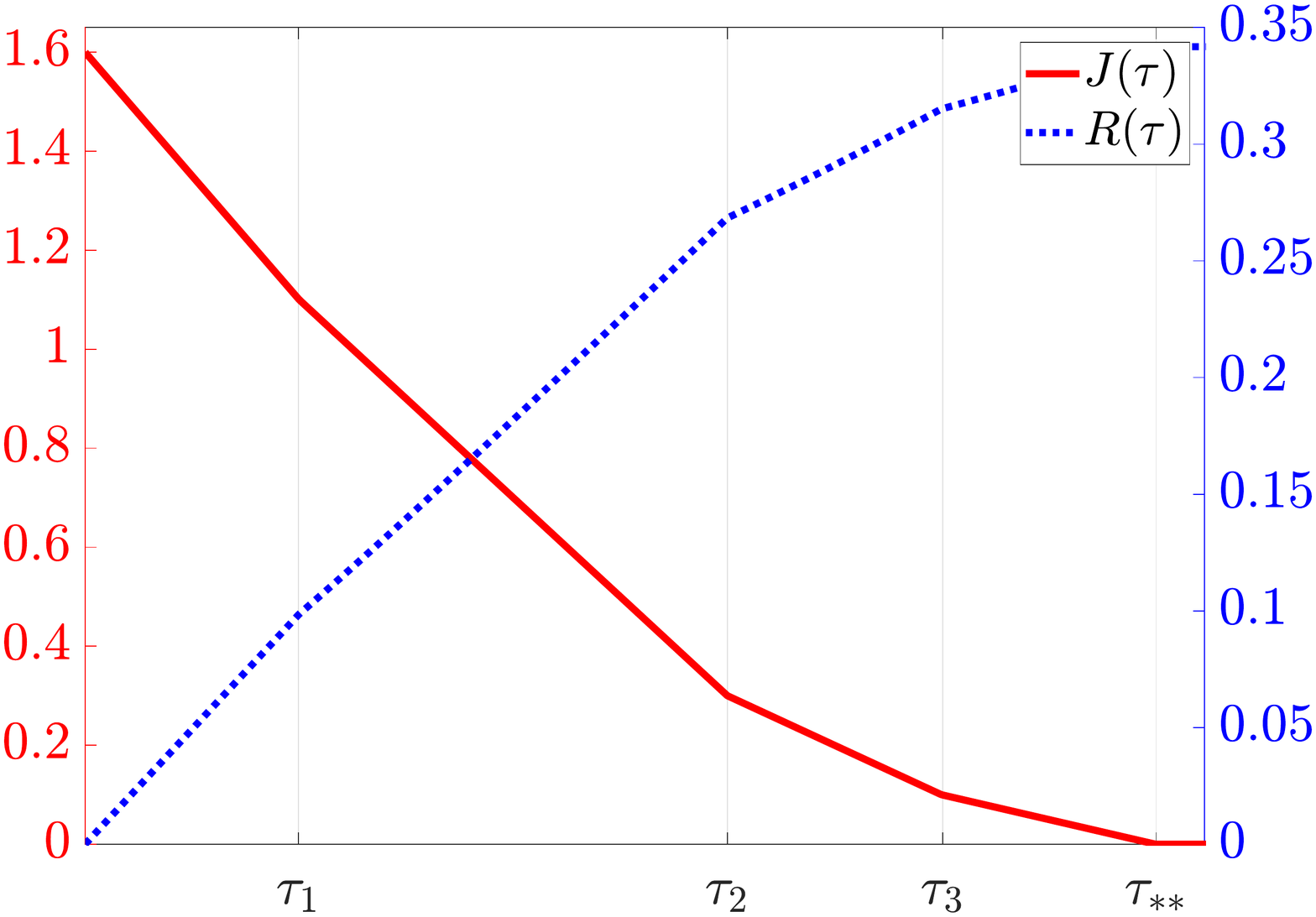}\\
\includegraphics[trim={2cm 1cm 1.2cm 0.5cm},clip, height=4cm, width=0.49\textwidth]{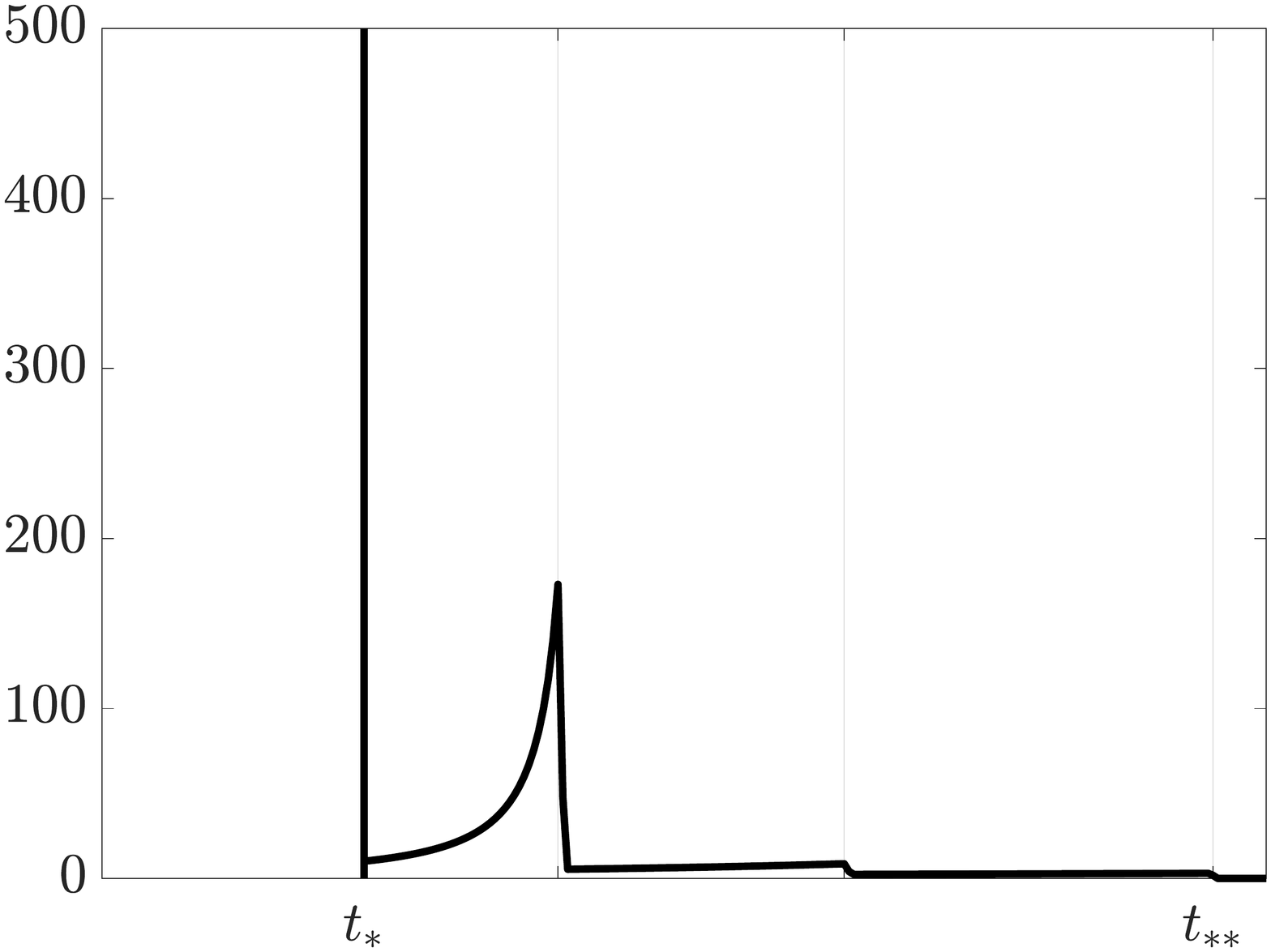}
\hfill
\includegraphics[trim={2cm 1cm 1.2cm 0.5cm},clip, height=4cm, width=0.49\textwidth]{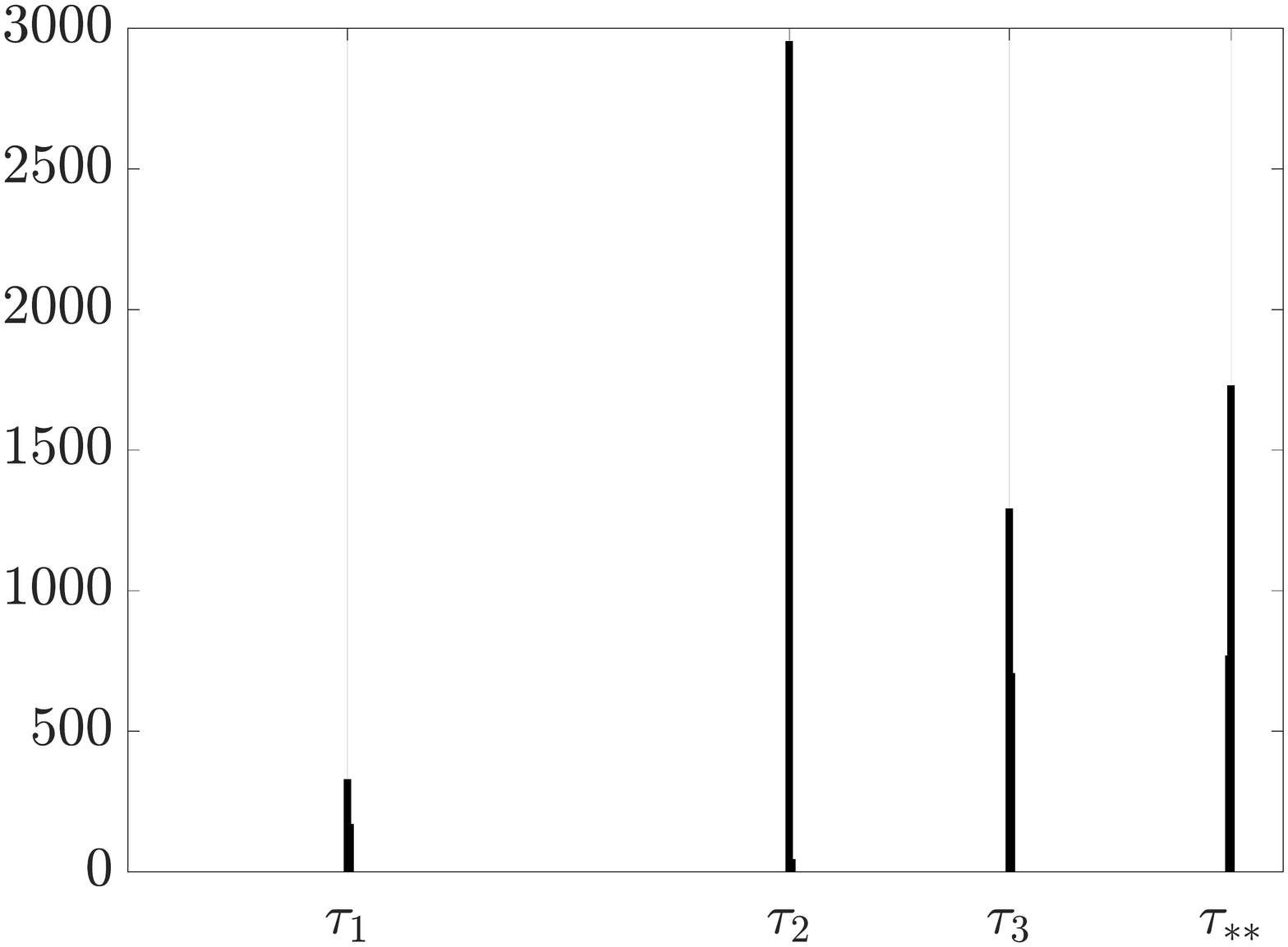}
\caption{Residual, regularizer (\textbf{top}) and spectra (\textbf{bottom}) of $u_t$ (\textbf{left}) and $v_\tau$ (\textbf{right})\label{fig:res_reg_spectra_deconv}}
\end{figure}

The residuals and regularizers of the solution paths, which are shown in the top row of \fref{fig:res_reg_spectra_deconv}, clearly reflect this behavior by having kinks at the critical times. Note that~$R(t)$ and~$J(t)$ indeed jump at~$t_*$ where the forward solution is not unique. Furthermore,~$R(\tau)$ and~$J(\tau)$ are piecewise linear in~$\tau$, as expected. Also the spectra, which are defined the~$1$-norm of the spectral representations~$\phi_t$ and~$\varphi_\tau$ and are depicted in the bottom row of \fref{fig:res_reg_spectra_deconv}, match our analytic results~\eref{eq:spectrum_1h} and~\eref{eq:spectrum_qu} since they posses numerical~$\delta$-peak at~$t_*$ or at the four critical times, respectively. In particular, we see that~$\phi_t$ does not have any atoms for~$t\neq t_*$. Note that the height of the spectral peaks is not informative since the measure at these points is given by an multiple of an Dirac measure which has ``infinite height''.

\subsection{\LB{Total variation scale space}}

Next, we turn to the \eref{mod:ROF} model and the variant with non-squared~$L^2$-norm, respectively. The data~$f$ is given by the ``Barbara'' image and is shown in the top right corner of \fref{fig:filters}. The top row of \fref{fig:res_reg_spectra_denoise} shows the residuals and regularizers of~$u_t$ and~$v_\tau$, respectively. We can observe that there is a positive~$t_*$ and that there are no kinks, meaning there is no visible piecewise behavior of the solution paths. The magnitudes of the spectral representations are given in the bottom row of \fref{fig:res_reg_spectra_denoise}. Note that both spectra, again defined as 1-norm of the spectral representations, behave very regular and do not show any numerical delta peaks. However, the spectrum of $\varphi_\tau$ contains much more information, being encoded in two elevations that are marked in red (dotted) and blue (dashed).

The top row in \fref{fig:filters} shows the corresponding spectral components~$\varphi_\tau$ integrated with respect to~$\tau$ over the red and blue area, respectively (cf.~\eref{eq:filtering}). This procedure can be viewed as band-pass filtering with respect to the nonlinear frequency decomposition~$\varphi_\tau$ and allows to extract and manipulate patterns and textures from the original image. In our example, these images correspond to differently oriented stripe patterns on the table cloth and Barbara's clothing. The spectrum of~$\phi_t$, however, cannot be used for this task since the only two significant parts of the spectrum -- marked in the same fashion -- correspond to very fine and fine structures (cf. second row of \fref{fig:filters}) but do not separate different textures. We have the suspicion that this behavior is explained by the closing remarks of \sref{sec:singular_vectors} according to which the~$\tv$-model with $\expd=1$ is scale-invariant on eigenfunctions in 2D. Indeed further numerical experiments indicate that the one-dimensional ROF model with non-squared data term is capable of capturing different scales. 

Another popular filtering procedure is high-pass and low-pass filtering which corresponds to keeping only the frequency components beyond or until a threshold frequency. The last two rows of \fref{fig:filters} show the corresponding filtered images using the spectra of~$\varphi_\tau$ and~$\phi_t$, respectively. Here, both methods succeed equally well in separating texture and objects. Regarding, high and low-pass filtering, it can be considered a slight advantage of the spectral representation generated by the scale and contrast-invariant model that the magnitude of the spectrum decreases more rapidly and that textures seem to be concentrated more compactly in the spectrum. This can make automatic filtering easier and more robust. 

\begin{figure}[h!]
\centering
\includegraphics[trim={2cm 1cm 1.2cm 0.5cm},clip, height=4cm, width=0.49\textwidth]{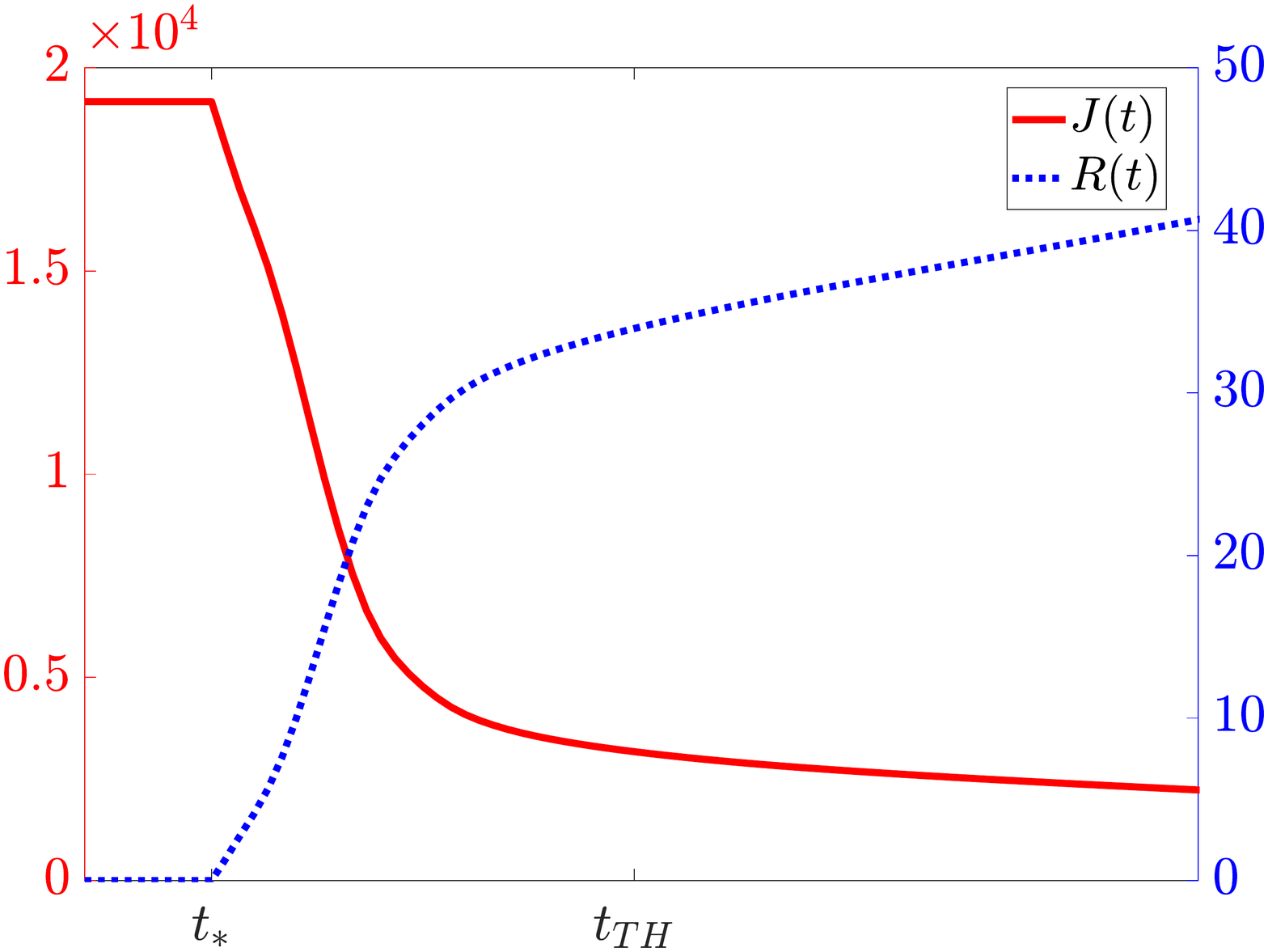}
\hfill
\includegraphics[trim={2cm 1cm 1.2cm 0.5cm},clip, height=4cm, width=0.49\textwidth]{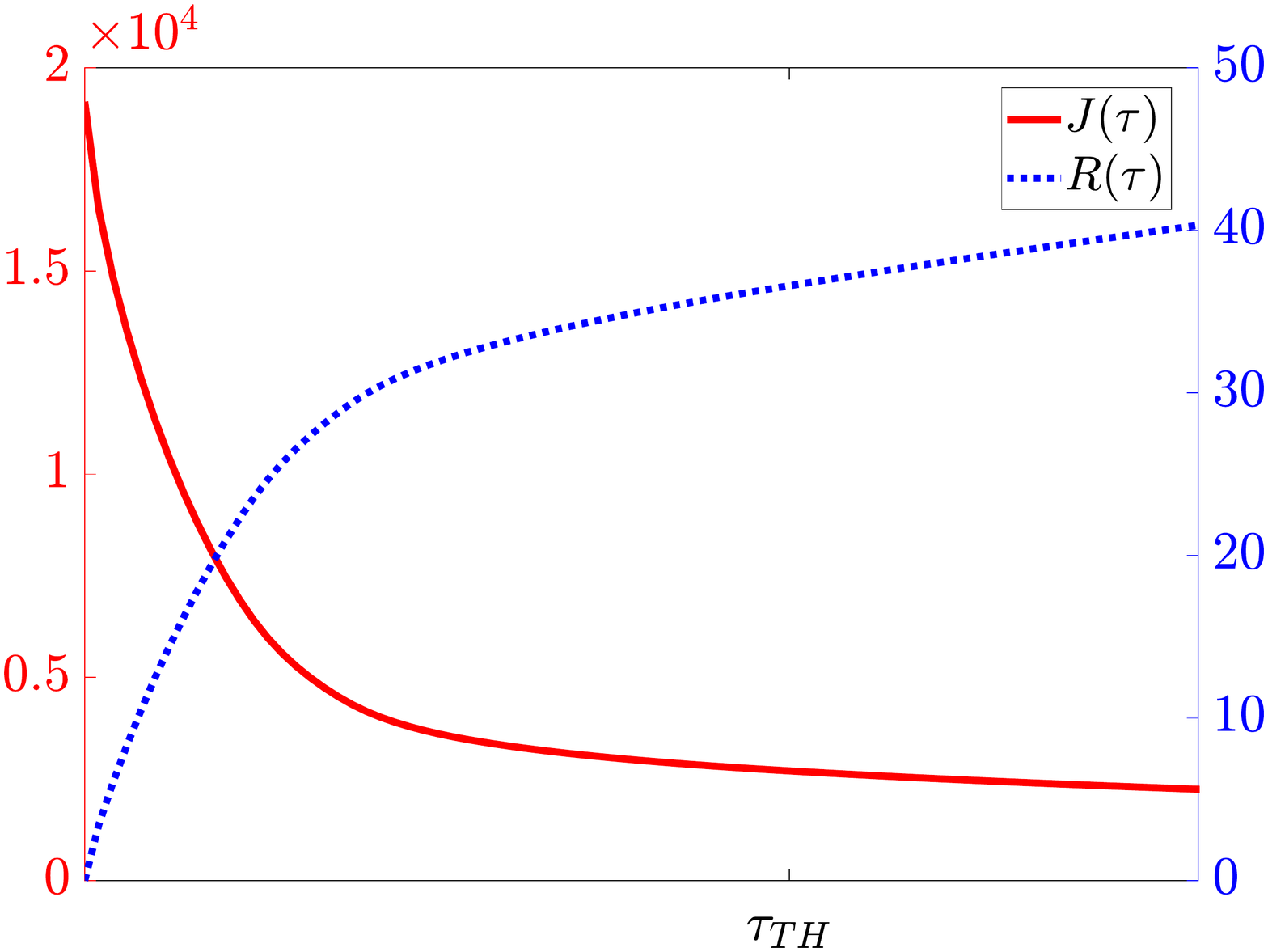}\\
\includegraphics[trim={2cm 1cm 1.2cm 0.5cm},clip, height=4cm, width=0.49\textwidth]{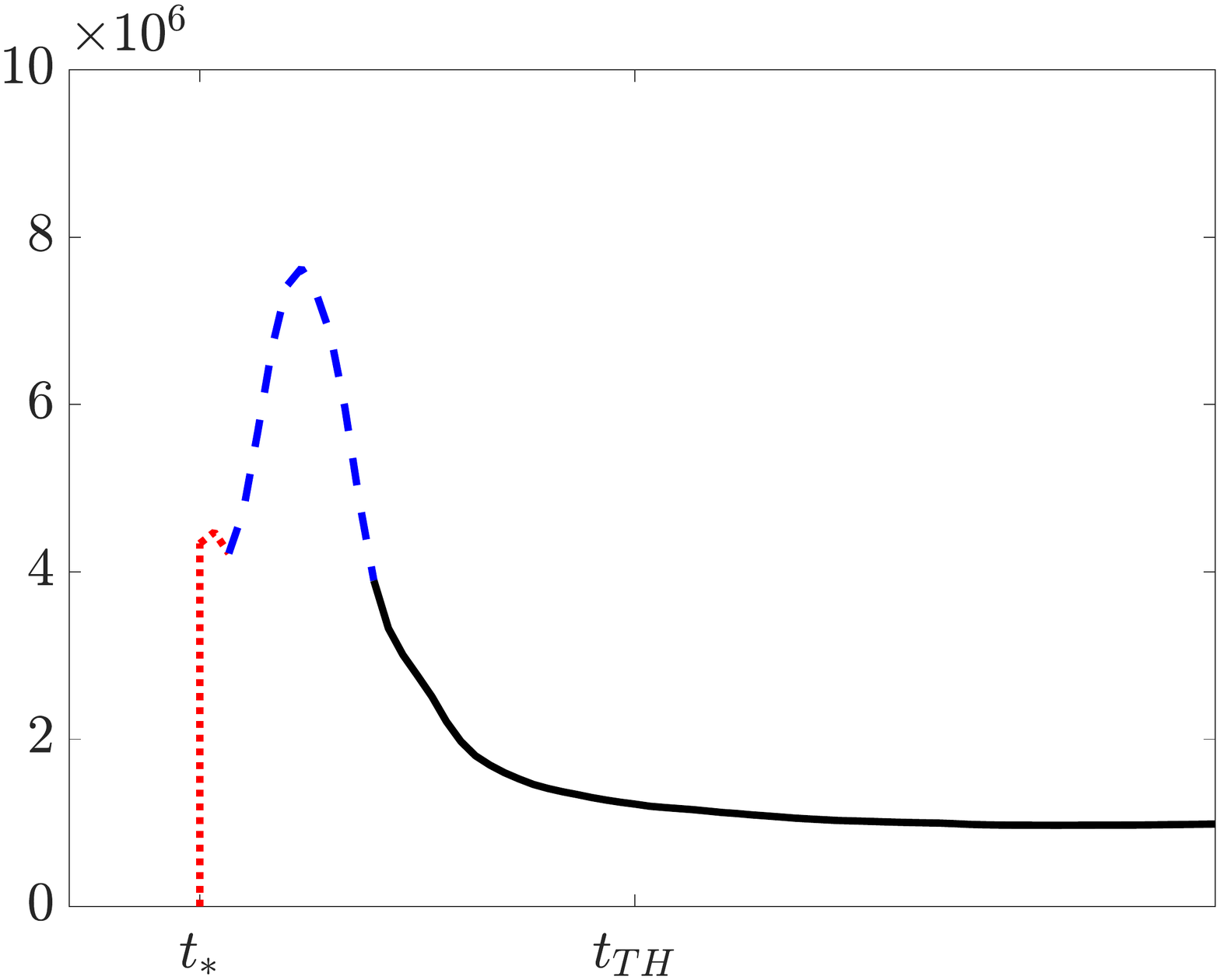}
\hfill
\includegraphics[trim={2cm 1cm 1.2cm 0.5cm},clip, height=4cm, width=0.49\textwidth]{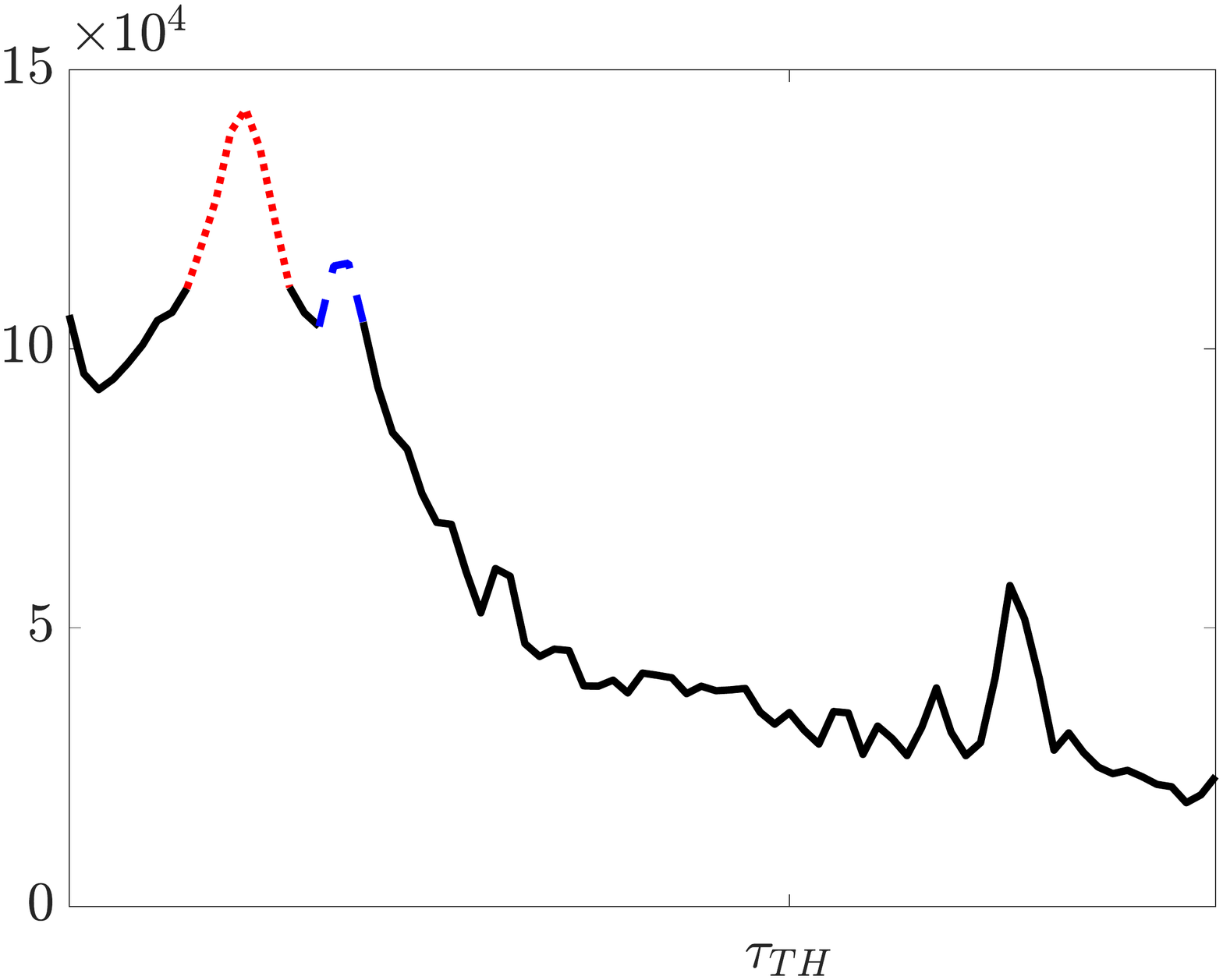}
\caption{Residual, regularizer (\textbf{top}) and spectra (\textbf{bottom}) of $u_t$ (\textbf{left}) and $v_\tau$ (\textbf{right})\label{fig:res_reg_spectra_denoise}}
\end{figure}


\begin{figure}[h]
\newcolumntype{C}{>{\centering\arraybackslash} m{4.5cm} }
\centering
\begin{tabular}{m{.2cm}CCC}
 & \textbf{Red} band-pass & \textbf{Blue} band-pass & Original image\\ 
$\varphi_\tau$ & \includegraphics[width=4.5cm]{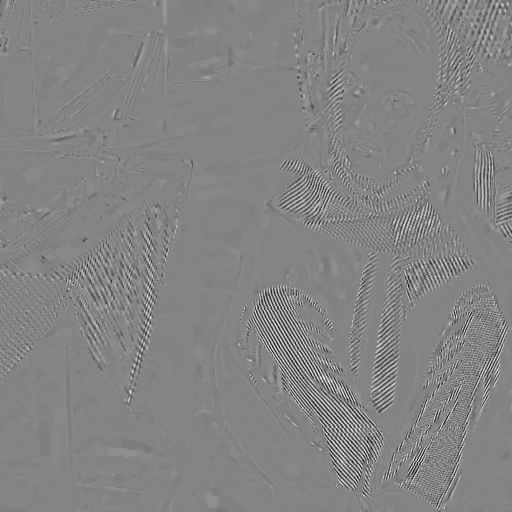} & \includegraphics[width=4.5cm]{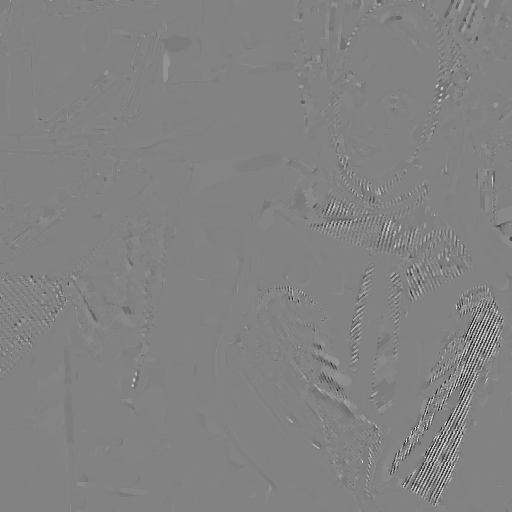} &\includegraphics[width=4.5cm]{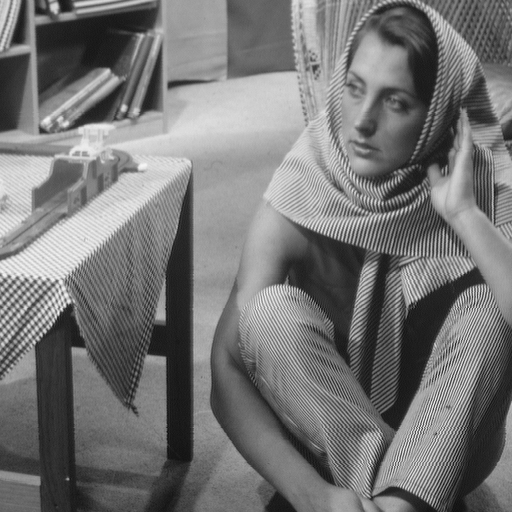}\\ 
$\phi_t$ & \includegraphics[width=4.5cm]{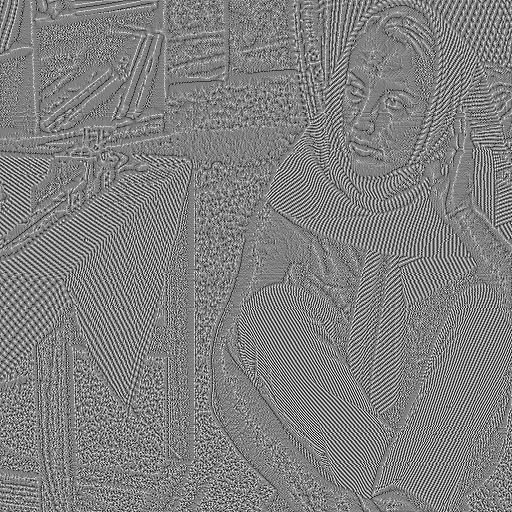}
 & \includegraphics[width=4.5cm]{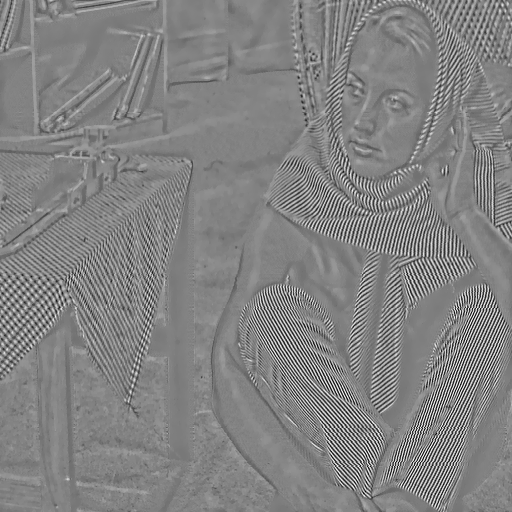}& \\ 
 & High-pass & Low-pass &\\ 
$\varphi_\tau$ & \includegraphics[width=4.5cm]{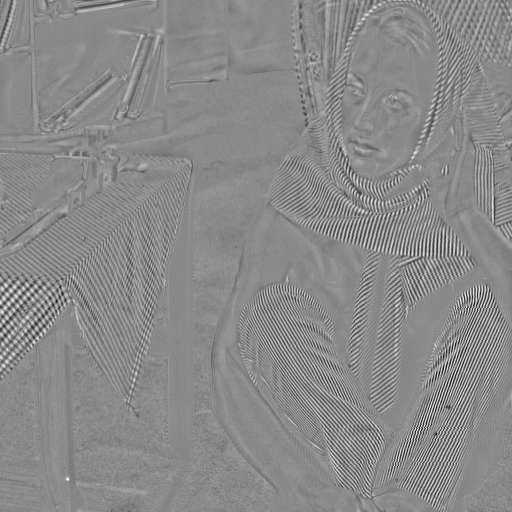} & \includegraphics[width=4.5cm]{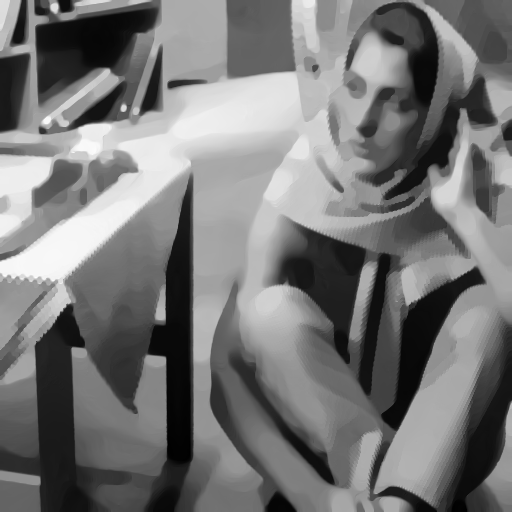} &\\ 
$\phi_t$ & \includegraphics[width=4.5cm]{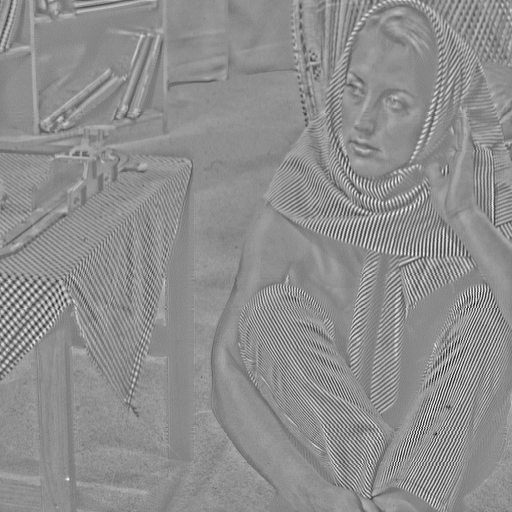} & \includegraphics[width=4.5cm]{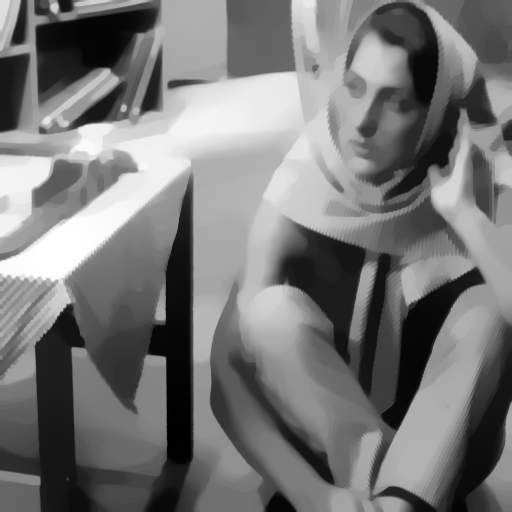}&
\end{tabular} 
\caption{Several filters applied to spectral representations $\varphi_\tau$ and~$\phi_t$\label{fig:filters}}
\end{figure}

\section*{Conclusion}
We have analyzed a family of variational regularization functionals with different powers of the data fidelity and regularization terms, among which the model with quadratic fidelity and absolutely one-homogeneous regularization stands out as the ``standard choice''. Apart from trivial solutions -- which are achieved for very small, respectively, large values of the regularization parameter -- all models generate the same set of minimizers. Therefore, simply aiming at finding a regular approximate solution to the inverse problem \eref{eq:IP}, no specific weighting can be preferred over others. However, if one is interested in the whole solution path and derivatives thereof with respect to the regularization parameter, the choice of the specific weighting becomes relevant. In particular, we have argued why it is necessary to choose the standard weighting in order to obtain nonlinear spectral decompositions. Furthermore, the failure of the contrast-invariant methods to decompose a linear combination of eigenvectors shows that enforcing consistency on a single eigenvector is not enough to define a meaningful spectral representation of arbitrary data.
    
\clearpage    
\section*{Some open questions}  
We conclude this work by pointing out some interesting open questions that are subject to future research.
\begin{enumerate}
\item It is an interesting question whether and how our results connect with generalized Cheeger sets (cf.~\cite{pratelli2017generalized}). It is well-known that a convex set is calibrable if and only if it is a Cheeger set in itself. Furthermore, we have seen that the extinction time of a calibrable set~$\Omega$ under~$\tv$ with data term $\frac{1}{\expd}\norm{u-f}^\expd_{L^2}$ is given by~$|\Omega|^\frac{\expd}{2}/P(\Omega)$ which is precisely the inverse Cheeger constant if~$\Omega$ is a generalized Cheeger set, i.e. a minimizer of ${P(E)}/{|E|^m}$ among all sets $E\subset\Omega$ with~$m\defi\expd/2$, where usually $1-1/n<m<1$ is assumed which corresponds to $2-2/n<\expd<2$. 
\item Furthermore, a relevant open point is to find sufficient conditions for~$\hat{\tau}>0$, meaning that~$v_\tau$ is affine on an interval~$(0,\hat{\tau})$. We suspect that the necessary condition from Proposition~\ref{prop:hat_tau} could also be sufficient but a proof is still pending.
\item Related to the former point is the well-definedness of~$\varphi_\tau$ as a Radon measure for general data. Certainly, a piecewise affine behavior of the solution path guarantees this but this does not occur, in general. However, we have the hope that formula~\eref{eq:moreau} can be used to deduce the regularity of~$v_\tau$ from the regularity of the boundary of the convex set~$K$. 
\end{enumerate}

\ack
This work was supported by the European Union’s Horizon 2020 research and innovation programme under the Marie Sk\l{}odowska-Curie grant agreement No 777826 (NoMADS). The authors acknowledge further support by ERC via Grant EU FP 7—ERC Consolidator Grant 615216 LifeInverse. Furthermore, we would like to express our gratitude to Julian Rasch for his valuable comments which improved our manuscript. Most of this study has been carried out while the authors where affiliated with the University of M\"{u}nster.

\section*{References}
\begingroup
\renewcommand{\section}[2]{}%
\bibliography{bibliography}
\endgroup

\section*{Appendix}
\setcounter{section}{0}
\renewcommand{\thesection}{\Alph{section}}
\numberwithin{equation}{section}
\section{Subdifferentials and absolutely one-homogeneous Functionals}\label{sec:notation}

We say that the convex functional~$J:\X\to\R\cup\{+\infty\}$ is absolutely one-homogeneous if $J(cu)=|c|J(u)$ holds for all~$c\in\R$ and~$u\in\X$.
The subdifferential of~$J$ in~$u\in\dom(J)$ is defined as
\begin{equation}
\partial J(u)\defi\left\lbrace p\in\dualX\st J(u)+\langle p,v-u\rangle\leq J(v)\;\forall v\in\X\right\rbrace
\end{equation}
and can be simplified to  
\begin{equation}\label{eq:subdiff_gen}
\partial J(u)=\left\lbrace p\in\dualX\st \langle p,v\rangle\leq J(v)\;\forall v\in\X,\,\langle p,u\rangle=J(u)\right\rbrace
\end{equation}
since~$J$ is absolutely one homogeneous~\cite{burger2016spectral}. Here,~$(\dualX,\norm{\cdot}_\dualX)$ is the dual space of~$\X$ and~$\langle\cdot,\cdot\rangle$ denotes the dual pairing of~$\dualX$ and~$\X$ which can be identified with the inner product if~$\X$ is a Hilbert space. Note that $J(u)\neq 0$ implies~$0\not\in\partial J(u)$. A special role is played by the subdifferential in zero 
\begin{align}
\partial J(0)=\left\lbrace p\in\dualX\st\langle p,v\rangle\leq J(v)\;\forall v\in\X\right\rbrace.
\end{align}
It allows to write \eref{eq:subdiff_gen} more compactly as
\begin{align}\label{eq:subdiff}
\partial J(u)=\left\lbrace p\in\partial J(0)\st\langle p,u\rangle=J(u)\right\rbrace.
\end{align}
Note that \eref{eq:subdiff} shows that~$\partial J$ is positively zero-homogeneous as set-valued map, meaning that $\partial J(cu)=\partial J(u)$ for all $c>0$. 

Furthermore, we remind the reader that~$J$ can be written as the convex conjugate of the characteristic function of~$\partial J(0)$: 
\begin{align}\label{eq:J_dual}
J(u)=\chi_{\partial J(0)}^*(u)=\sup_{q\in\partial J(0)}\langle q,u\rangle,\quad u\in\X.
\end{align}
Here we used the characteristic function $\chi_M$ of an arbitrary set $M$, which is defined as
\begin{align}\label{eq:char_func}
\chi_M(x)\defi
\begin{cases}
0,\quad &x\in M,\\
\infty,\quad &x\notin M.
\end{cases}
\end{align}
For the sake of completeness we add the (similar but different) definition of the indicator function $\mathbf{1}_\Omega$ of a set $\Omega\subset\R^n$:
\begin{align}\label{eq:ind_func}
\mathbf{1}_\Omega(x)\defi
\begin{cases}
1,\quad&x\in\Omega,\\
0,\quad&x\notin\Omega.
\end{cases}
\end{align}
Since~$J$ is absolutely one-homogeneous, it is a semi-norm on a subspace of $\X$ and it holds~(cf.~\cite{burger2016spectral})
\begin{subequations}\label{eq:properties_J}
\begin{align}
&J(u)\geq0,\quad&&\forall u\in\X\label{eq:po_J}\\
&J(u+v)\leq J(u)+J(v)\quad&&\forall u,v\in\X\label{eq:triangle_J}\\
&J(u+v_0)=J(u),\quad &&\forall u\in\X,v_0\in\calN(J),\label{eq:kernel_J}\\
&p\in\partial J(u)\implies \langle p,v_0\rangle=0,\quad&&\forall v_0\in\calN(J)\label{eq:subgrad_kernel_J}.
\end{align}
\end{subequations}
Finally, from~\eref{eq:subdiff} it follows that the symmetric Bregman distance is non-negative, i.e.,
\begin{align}\label{eq:pos_bregman}
\langle p-q,u-v\rangle\geq 0,\quad p\in\partial J(u),\;q\in\partial J(v).
\end{align}

\section{Generalized orthogonal projections}\label{sec:orth_proj}
\begin{lemma}
The set 
$$ \calN(J) \defi \{ u~\in \X\st J(u) = 0\}~$$
is a closed linear subspace of~$\X$ in the weak$^*$ and strong topology.
\end{lemma} 
\begin{proof}
From the absolute homogeneity we obtain for~$J(u) = 0$ also~$J(c u) = |c| J(u) =0$ for all $c\in\R$. Moreover, the triangle inequality \eref{eq:triangle_J} implies for~$J(u_i)=0$,~$i=1,2$, that
$$ 0 \leq J(u_1 + u_2) \leq J(u_1) + J(u_2) = 0.~$$
Thus, linear combinations of elements in~$\calN(J)$ remain in~$\calN(J)$. Now assume~$u_k$ is a weakly$^*$ convergent sequence in~$\calN(J)$, then the weak$^*$ lower semi-continuity implies for the limit~$u$
$$ 0 \leq J(u) \leq \liminf_{k\to\infty} J(u_k) = 0,~$$
hence~$u \in \calN(J)$. Thus,~$\calN(J)$ is a weakly$^*$ closed subspace which also implies closedness in the strong topology. 
\end{proof} 

The following definition is a generalization of the orthogonal complement in Hilbert spaces to our Banach space setting. 
\begin{definition}\label{defi:orth_compl}
For a subset~$U\subset\X$ we define the~$\fwd$\emph{-orthogonal complement} of~$U$ in~$\X$ as
\begin{align}
U^{\perp,\fwd}\defi\left\lbrace v\in\X\st\langle\fwd v,\fwd u\rangle=0,\quad\forall u\in U\right\rbrace.
\end{align}
\end{definition}

Note that, owing to Assumption~\ref{assmpt:w*-w-c}, the $\fwd$-orthogonal complement of $U$ is a weakly$^*$ closed subspace of $\X$.

\begin{thm}\label{thm:ex_long_time_prob}
The $\fwd$-orthogonal projection~$\calP^\fwd$ given by \eref{eq:orth_proj} is well-defined.
\end{thm}
\begin{proof}
Let~$(u_k)\subset\calN(J)$ be a minimizing sequence for the problem. Hence,~$(u_k)$ is bounded with respect to~$\norm{\cdot}_\fwd$ and by Assumption~\ref{assmpt:A-norm} also in~$\norm{\cdot}_\X$. Thus, using Banach-Alaoglu and that~$\calN(J)$ is weakly$^*$ closed, up to a subsequence, the sequence~$(u_k)$ weakly$^*$ converges to some~$u\in\calN(J)$. Furthermore, the sequence~
$\fwd u_k$ converges weakly to~$\fwd u$ by Assumption~\ref{assmpt:w*-w-c} such that the weak lower semi-continuity of~$\norm{\cdot}_\H$ shows that~$u$ is a minimizer. Uniqueness can be established by observing that the second variation of the functional under optimization is positive definite since~$\fwd$ is injective on~$\calN(J)$. 
\end{proof}

\begin{prop}\label{prop:properties_proj}
Let~$\calP^\fwd:\H\rightarrow\X$ be as before. It holds
\begin{enumerate}
\item (Range and Nullspace)~$\ran(\calP^\fwd)=\calN(J)$ and~$\calN(\calP^\fwd)=(\fwd \calN(J))^\perp$´,
\item (Idempotence)~$\calP^\fwd(\fwd\calP^\fwd(f))=\calP^\fwd(f),\quad\forall f\in\H$,
\item (Orthogonality)~$\langle f-\fwd\calP^\fwd(f),\fwd v\rangle=0,\quad\forall f\in\H,\,v\in\calN(J)$,
\item (Linearity)~$\calP^\fwd:\H\to\X$ is linear and bounded,
\item (Self-adjointness)~$\langle f,\fwd\calP^\fwd(g)\rangle=\langle\fwd\calP^\fwd(f),g\rangle,\quad\forall f,g\in\H.$
\end{enumerate}
\end{prop}
\begin{proof}
First note that per definitionem~$\ran(\calP^\fwd)\subset\calN(J)$. The converse inclusion also holds since any~$u\in\calN(J)$ can be written as~$u=\calP^\fwd(\fwd u)$. Now let~$f \in (\fwd\calN(J))^{\perp}$, then for each~$u \in \calN(J)$
$$ \Vert Au - f \Vert_\H^2 = \Vert Au \Vert_\H^2 + \Vert  f \Vert_\H^2 \geq\Vert  f \Vert_\H^2~$$
with equality for~$u=0$, i.e.,~$ \calP^\fwd (f) = 0$. Assume vice versa ~$ \calP^\fwd (f) = 0$, then for~$u \in \calN(J)$ and~$\varepsilon \in \R$ we have
$$ 0 \leq \Vert \varepsilon \fwd u - f\Vert_\H^2 - \Vert f \Vert_H^2 = \varepsilon^2 \Vert \fwd u \Vert_\H^2 - 2 \varepsilon \langle \fwd u , f \rangle .~$$
In the limit~$\varepsilon \rightarrow 0$ we find~$\langle \fwd u , f \rangle = 0$ taking into account the arbitrary sign of~$\varepsilon$. Idempotence is trivial by observing that~$u=\calP^\fwd(f)$ satisfies {$\norm{\fwd u-\fwd\calP^\fwd(f)}_\H=0$}. Orthogonality is obtained by an adaption of the standard proof in the Hilbert space setting. Defining~$a\defi f-\fwd\calP^\fwd(f)\in\H$ it holds for all~$w\in\H$
$$\left\Vert a-\frac{\langle a,w\rangle}{\norm{w}_\H^2}w\right\Vert_\H^2=\norm{a}_\H^2-\frac{\langle a,w\rangle^2}{\norm{w}_\H^2}.$$
If we now set 
$$u\defi\calP^\fwd(f)+\frac{\langle a,\fwd v\rangle}{\norm{v}_\H^2}v\in\calN(J)$$
for~$v\in\calN(J)$ and apply the first equality with~$w\defi\fwd v$, we infer that 
$$\norm{f-\fwd u}_\H^2=\norm{f-\fwd\calP^\fwd(f)}_\H^2-\frac{\langle f-\fwd\calP^\fwd(f),\fwd v\rangle^2}{\norm{\fwd v}_\H^2}.$$
Hence, since~$\calP^\fwd(f)$ is the minimizer in~\eref{eq:long_time_prob}, one can conclude that the scalar product has to vanish. Linearity follow from orthogonality. Since both~$\calN(\calP^\fwd)$ and~$\ran(\calP^\fwd)$ are closed and~$\calP^\fwd$ possesses all properties of a projection, it is straightforward to show that~$\calP^\fwd$ is a closed operator and, hence, bounded by the closed graph theorem. Also self-adjointness follows directly from orthogonality.
\end{proof}

\end{document}